\documentclass[11pt]{article}
\usepackage{cite}

\usepackage{graphicx} 
\usepackage{amssymb,latexsym,amsfonts,amsmath,amsthm}
\usepackage{hyperref}
\usepackage[dvipsnames]{xcolor}
\usepackage{tikz}
\usetikzlibrary{graphs}
\usetikzlibrary{graphs.standard}
\usetikzlibrary{shapes.multipart,shapes.geometric,topaths,calc, positioning, decorations.pathreplacing, calligraphy}
\usepackage{ragged2e}
\usepackage{blindtext}
\usepackage{soul}
\usepackage{array}
\usepackage{tikz,lipsum,lmodern}
\usetikzlibrary{graphs}
\usetikzlibrary{matrix}
\usetikzlibrary{graphs.standard}
\usepackage{hyperref}
\usepackage{mathalpha}
\usepackage{array}
\usepackage{float}
\usepackage[most]{tcolorbox}
\usepackage[paperwidth=8.5in,paperheight=10.5 in]{geometry}
\usepackage{caption}
\usepackage{subcaption}

\theoremstyle{plain}
\newtheorem{thm}{Theorem}
\newtheorem{cor}[thm]{Corollary}
\newtheorem{lem}[thm]{Lemma}
\newtheorem{prop}[thm]{Proposition}
    
\newtheorem{question}[thm]{Question}

\newtheorem{obs}[thm]{Observation}
\newtheorem{conj}[thm]{Conjecture}

\newtheorem*{rem}{Remark}

\theoremstyle{definition}
\newtheorem{defn}[thm]{Definition}

\newcommand{\leaky}[2]{Z_{(#1)}(#2)}
\newcommand{\psdl}[2]{Z_{(#1)}^{+}(#2)}
\newcommand{\psdlfort}{F_{(\ell)}^+}

\definecolor{backblue}{RGB}{230,245,245}
\definecolor{titleturq}{RGB}{54,89,89}

\def\hyperset{\pgfqkeys{/tikz/hyper}}
\tikzset{hypher/.code=\hyperset{#1}}

\hyperset{
  set color/.code 2 args={\colorlet{tikz@hyper@dimen@#1}{#2}},
  set color={0}{black},
  set color={1}{blue!50!black},
  set color={2}{red},
  set color={3}{green},
  set color={4}{yellow!80!black}}
\hyperset{
  set dimens/.style args={#1:(#2)}{
    dimen #1/.style={/tikz/shift={(#2)}}},
  set dimens=0:(0:0),
  set dimens=1:(right:1),
  set dimens=2:(up:1),
  set dimens=3:(30:.75),
  set dimens=4:(180+70:.5),
  every hyper node/.style 2 args={%
    shape=circle,
    inner sep=+0pt,
    fill,
    draw,
    minimum size=+3pt,
    color=tikz@hyper@dimen@#1,
    label={[hyper/label #1/.try, hyper/dimen #1 style/.try]#1-#2}
  },
  every hyper edge/.style={draw},
  every hyper shift edge/.style={->,tikz@hyper@dimen@#1!80},
  every normal hyper edge/.style={<->,tikz@hyper@dimen@#1!40},
}

\newcommand*{\hyperhyper}[3]{
  \foreach \dimension in {#3,...,\the\numexpr#1+1\relax} {
    \edef\newNumber{\the\numexpr#2+\dimension\relax}
    \node[hyper/every hyper node/.try={\dimension}{\newNumber}, hyper/dimen \dimension node/.try, hyper/dimen \dimension\space style/.try] at ([hyper/dimen \dimension] #1-#2) (\dimension-\newNumber) {};
    \path (#1-#2) edge[hyper/every hyper edge/.try=\dimension, hyper/every hyper shift edge/.try=\dimension, hyper/dimen \dimension\space style/.try] (\dimension-\newNumber);
    \ifnum\newNumber>\dimension\relax
      \foreach \oldShift in \currentTransform {
        \if\relax\detokenize\expandafter{\oldShift}\relax\else
          \path (\dimension-\newNumber) edge[hyper/every hyper edge/.try=\dimension, hyper/every normal hyper edge/.try=\dimension, hyper/dimen \dimension\space style/.try] (\dimension-\the\numexpr\newNumber-\oldShift\relax);
        \fi
      }
    \fi
    \edef\currentTransform{\dimension,\currentTransform}%
    \ifnum\dimension<#3\relax
      \edef\temp{{\dimension}{\the\numexpr#2+\dimension\relax}{#3}}
      \expandafter\hyperhyper\temp
    \fi
  }
}

\tikzset{
  @only for the animation/.style={
    hyper/dimen #1 style/.style={opacity=0}}}

\title{Leaky Positive Semidefinite Forcing on Graphs}
\author{Olivia Elias \thanks{University of Colorado - Colorado Springs, Colorado Springs, CO, USA},
Ian Farish \thanks{California State Polytechnic University - Pomona, Pomona, CA, USA},
Emrys King \thanks{Pomona College, Claremont, CA, USA},
Josh Kyei \thanks{Morehouse College, Atlanta, GA, USA},
Ryan Moruzzi, Jr \thanks{Corresponding Author, Department of Mathematics, California State University East Bay, Hayward, CA, USA (ryan.moruzzi@csueastbay.edu)}}

\definecolor{cmured}{RGB}{183,1,48}
\definecolor{cmunavy}{RGB}{48,1,183}

\begin{document}

\maketitle
\begin{abstract}
      We introduce $\ell$-leaky positive semidefinite forcing and the $\ell$-leaky positive semidefinite number of a graph, $\psdl{\ell}{G}$, which combines the positive semidefinite color change rule with the addition of leaks to the graph. Furthermore, we determine general properties of $\psdl{\ell}{G}$ and $\psdl{\ell}{G}$ for various graphs, including path graphs, complete graphs, wheel graphs, complete bipartite graphs, trees, hypercubes, and prisms. We also define $\ell$-leaky positive semidefinite forts with the purpose of unveiling differences between $\ell$-leaky standard forcing and $\ell$-leaky positive semidefinite forcing.
\end{abstract}



\section{Introduction}

\quad \quad Zero forcing is an iterative vertex-coloring process on a graph $G$ that begins with a subset of vertices of $G$ colored blue and continues with the iterative application of a color change rule until no further forces can occur. In standard zero forcing, the color change rule is as follows: If a blue vertex has a unique non-blue neighbor, then the non-blue vertex can be ``forced" blue by the blue vertex. Forces are applied over the course of multiple rounds until there are no other forces that can occur. If an initial set of blue vertices is capable of forcing \textit{all} vertices in $G$ blue, then that initial subset of vertices is called a zero forcing set. The size, in cardinality, of a smallest zero forcing set of $G$ is the zero forcing number of $G$, denoted $Z(G)$.

Since 2008, stemming from its connection with the inverse eigenvalue problem of a graph \cite{AIMMINIMUMRANKWorkgroup}, zero forcing has grown into its own area of interest, in part due to its multitude of applications in mathematics, physics, and computer science. Zero forcing can be intuitively used to model the spread of information throughout a network, and theoretical results in this field have been useful in applications to quantum control \cite{QuantumControl} and search engines \cite{SearchTrees}.

With the growth of zero forcing on graphs as an area of research, variations of zero forcing have garnered much interest, particularly those with connection to the inverse eigenvalue problem of a graph (IEP-G), such as \textit{skew zero forcing} (see \cite{IEPGZF} and references therein). One variation with a connection to the IEP-G and the variation we focus on is \textit{positive semidefinite zero forcing}. In \cite{PSDOrigins}, positive semidefinite forcing was introduced, which essentially allows a single blue vertex to perform multiple forces within a single round; a formal definition of positive semidefinite forcing is provided in Section \ref{sec:Prelims}. As in standard zero forcing, a positive semidefinite forcing set is a set of initial blue vertices that is able to force an entire graph blue according to the positive semidefinite forcing color change rule. The positive semidefinite forcing number of a graph $G$, $Z_+(G)$, is the size of a smallest positive semidefinite forcing set of $G$. 

Motivated by the application of zero forcing to model the minimum number of sensors needed to observe an entire network, in \cite{DillmanKenter}, authors S. Dillman and F. Kenter asked the question ``What if there
is a leak in the network?'' where leaks hinder the observation process. Having defined leaks in a graph as vertices that cannot perform forces, the new set of rules for leaky forcing on graphs has generated recent interest as a variant of zero forcing \cite{LeakyReslience,D-2LeakyCube,GeneralizationsLeaky}. Although a connection between leaky forcing and the inverse eigenvalue problem of a graph has not been determined, introducing leaks to hinder the forcing process is an interesting question in its own right, particularly due to the interactions of leaks and other variations of zero forcing. In this paper, we introduce $\ell$-leaky positive semidefinite forcing, which combines the positive semidefinite color change ruleset with the presence of leaks. Like earlier variations, we define a $\ell$-leaky positive semidefinite forcing set as a set of initial blue vertices that can be used for an entire graph blue using the positive semidefinite color change rule regardless of the inclusion of $\ell$ leaks. Furthermore, the $\ell$ -leaky positive semidefinite forcing number of a graph $G$, $\psdl{\ell}{G}$, gives the size of a smallest $\ell$-leaky positive semidefinite forcing set of $G$. See Section \ref{sec:Prelims} for formal definitions of the previous parameters. Our determination of the $\ell$-leaky positive semidefinite forcing number being the same as the $\ell$-leaky standard number for various graphs led us to define $\ell$-leaky positive semidefinite forts in hopes of uncovering structural properties that relate $\ell$-leaky standard and $\ell$-leaky positive semidefinite forcing numbers. These $\ell$-leaky positive semidefinite forts are an interesting avenue for further research, particularly with the recent interest in forts, zero forcing, and computation \cite{Forts-Brimkov2019,FortsPSD,FortsOG,Forts-Optimal}.  

The paper is organized as follows. In Section \ref{sec:Prelims}, we provide the definitions used throughout, including the formal definition of $\ell$-leaky positive semidefinite forcing. In Section \ref{sec:PropertiesLeakyForcing}, we prove various properties of $\ell$-leaky positive semidefinite forcing. In Section \ref{sec:GraphFamilies}, we provide $\psdl{\ell}{G}$ for various graph families including paths, cycles, complete graphs, trees, wheels, and complete bipartite graphs. In Section \ref{sec:CartesisnProducts}, we investigate the $\ell$-leaky positive semidefinite forcing number of the Cartesian product of two graphs including hypercubes and prism graphs. In Section \ref{sec:Forts}, we define $\ell$-leaky positive semidefinite forts and establish their connection with $\ell$-leaky positive semidefinite forcing sets in hopes of structurally connecting those sets with $\ell$-leaky standard forcing sets.

\section{Preliminaries}\label{sec:Prelims}
\quad \quad Let $n\in \mathbb{Z}_+ = \{0,1,2,...\}$ and $G$ be a simple, undirected, finite graph with vertex set $V(G) = \{0, 1, 2, ... , n-1\}$ and edge set $E(G) = \{(i,j) : i\ne j,~i,j\in V(G)\}$. The neighborhood of $v\in V(G)$ is given by $N_G(v) = \{u \in V(G) : (u,v) \in E(G)\}$ and the degree of a vertex $v \in V(G)$ is given by deg$(v) = |\{u \in V(G) : u \in N_G(v)\}|$ while the minimum degree of the graph is given by $\delta(G) = \min\{\text{deg}(v) : v\in V(G)\}$. 
 
 A subgraph of $G$, $H\subseteq G$, is a graph such that $V(H) \subseteq V(G)$ and $E(H)\subseteq E(G)$. Furthermore, for any subset of vertices $S = \{v_1, v_2, ... , v_j\}\subseteq V(G)$,  the induced subgraph of $G$, denoted $H = G[S]$, is a subgraph of $G$ with edge set satisfying the following: for all $i,j\in V(H)$, $(i,j)\in E(H)$ if and only if $(i,j) \in E(G)$. For any subset of vertices $S\subseteq V(G)$, the graph $G-S$ is the induced subgraph formed with vertices $V(G)\setminus S$. Graphs of interest in sections that follow will be the path graph on $n$ vertices, $P_n$, the cycle graph on $n$ vertices, $C_n$, the complete graph on $n$ vertices, $K_n$, the wheel graph on n vertices, $W_n$, the complete bipartite graph on $n,m$ vertices, $K_{n,m}$, and tree graphs $T$ on $n$ vertices.

We study an iterative color change process known as positive semidefinite zero forcing which is a variation of standard forcing. The overall goal of this iterative color change process is to start with an initial set of vertices of the graph colored blue and apply the color change rule to force other vertices of the graph blue in hopes of turning the entire graph blue. The positive semidefinite forcing process is highlighted in Figure \ref{fig:PSDforcingprocess} and is accentuated by breaking up the graph $G$ into non-blue connected components determined by the initial blue vertex set. The precise definition of the positive semidefinite (psd) forcing process is as follows.  

\begin{defn}\label{def:PSDforce}
     Let $B^+\subseteq V(G)$ be a subset of vertices initially colored blue and $W_{1},\dots ,W_{k}$ be the sets of vertices of the components of $G-B^+$. Considering $G[W_{i} \cup B^+]$, if for $u \in B^+$ and $w_{i} \in W$, $w_{i}$ is the only non-blue neighbor of $u$, then $u \rightarrow w_i$ will denote a \emph{psd force} and $w_i$ is colored blue. 
\end{defn}

\begin{defn}
    Let $B^+\subseteq V(G)$ be a subset of vertices initially colored blue. $B^+$ is a \emph{positive semidefinite forcing set} if, after iteratively applying the positive semidefinite forcing process, the entire graph is colored blue. 
\end{defn}

\begin{figure}[h!]
    \centering
    \begin{tikzpicture}[x=0.75pt,y=0.75pt,yscale=-1.25,xscale=1.25]

\draw    (110.6,192.2) -- (140.2,192.42) ;
\draw    (85,101.4) -- (110.2,101.4) ;
\draw    (110.2,101.4) -- (139.8,101.62) ;
\draw    (139.8,101.62) -- (162.04,101.55) ;
\draw    (139.8,101.62) -- (139.98,79.74) ;
\draw    (140.27,123.68) -- (139.8,101.62) ;
\draw    (58.2,101.4) -- (85,101.4) ;
\draw  [color={rgb, 255:red, 0; green, 0; blue, 0 }  ,draw opacity=1 ][fill={rgb, 255:red, 255; green, 255; blue, 255 }  ,fill opacity=1 ] (53.68,101.4) .. controls (53.68,98.93) and (55.7,96.93) .. (58.2,96.93) .. controls (60.7,96.93) and (62.72,98.93) .. (62.72,101.4) .. controls (62.72,103.87) and (60.7,105.87) .. (58.2,105.87) .. controls (55.7,105.87) and (53.68,103.87) .. (53.68,101.4) -- cycle ;
\draw  [fill={rgb, 255:red, 255; green, 255; blue, 255 }  ,fill opacity=1 ] (135.46,79.74) .. controls (135.46,77.27) and (137.48,75.27) .. (139.98,75.27) .. controls (142.48,75.27) and (144.5,77.27) .. (144.5,79.74) .. controls (144.5,82.21) and (142.48,84.21) .. (139.98,84.21) .. controls (137.48,84.21) and (135.46,82.21) .. (135.46,79.74) -- cycle ;
\draw  [fill={rgb, 255:red, 255; green, 255; blue, 255 }  ,fill opacity=1 ] (105.68,101.4) .. controls (105.68,98.93) and (107.7,96.93) .. (110.2,96.93) .. controls (112.7,96.93) and (114.72,98.93) .. (114.72,101.4) .. controls (114.72,103.87) and (112.7,105.87) .. (110.2,105.87) .. controls (107.7,105.87) and (105.68,103.87) .. (105.68,101.4) -- cycle ;
\draw  [color={rgb, 255:red, 0; green, 0; blue, 0 }  ,draw opacity=1 ][fill={rgb, 255:red, 255; green, 255; blue, 255 }  ,fill opacity=1 ] (157.52,101.55) .. controls (157.52,99.08) and (159.55,97.07) .. (162.04,97.07) .. controls (164.54,97.07) and (166.57,99.08) .. (166.57,101.55) .. controls (166.57,104.01) and (164.54,106.02) .. (162.04,106.02) .. controls (159.55,106.02) and (157.52,104.01) .. (157.52,101.55) -- cycle ;
\draw  [fill={rgb, 255:red, 255; green, 255; blue, 255 }  ,fill opacity=1 ] (135.75,123.68) .. controls (135.75,121.21) and (137.78,119.21) .. (140.27,119.21) .. controls (142.77,119.21) and (144.8,121.21) .. (144.8,123.68) .. controls (144.8,126.15) and (142.77,128.15) .. (140.27,128.15) .. controls (137.78,128.15) and (135.75,126.15) .. (135.75,123.68) -- cycle ;
\draw  [fill={rgb, 255:red, 255; green, 255; blue, 255 }  ,fill opacity=1 ] (80.48,101.4) .. controls (80.48,98.93) and (82.5,96.93) .. (85,96.93) .. controls (87.5,96.93) and (89.52,98.93) .. (89.52,101.4) .. controls (89.52,103.87) and (87.5,105.87) .. (85,105.87) .. controls (82.5,105.87) and (80.48,103.87) .. (80.48,101.4) -- cycle ;
\draw  [fill={rgb, 255:red, 0; green, 7; blue, 255 }  ,fill opacity=1 ] (135.27,101.62) .. controls (135.27,99.15) and (137.3,97.15) .. (139.8,97.15) .. controls (142.29,97.15) and (144.32,99.15) .. (144.32,101.62) .. controls (144.32,104.09) and (142.29,106.09) .. (139.8,106.09) .. controls (137.3,106.09) and (135.27,104.09) .. (135.27,101.62) -- cycle ;
\draw    (189.2,100.8) -- (223.8,100.99) ;
\draw [shift={(225.8,101)}, rotate = 180.31] [color={rgb, 255:red, 0; green, 0; blue, 0 }  ][line width=0.75]    (10.93,-3.29) .. controls (6.95,-1.4) and (3.31,-0.3) .. (0,0) .. controls (3.31,0.3) and (6.95,1.4) .. (10.93,3.29)   ;
\draw    (276.2,101) -- (301.4,101) ;
\draw    (249.4,101) -- (276.2,101) ;
\draw  [color={rgb, 255:red, 0; green, 0; blue, 0 }  ,draw opacity=1 ][fill={rgb, 255:red, 255; green, 255; blue, 255 }  ,fill opacity=1 ] (244.88,101) .. controls (244.88,98.53) and (246.9,96.53) .. (249.4,96.53) .. controls (251.9,96.53) and (253.92,98.53) .. (253.92,101) .. controls (253.92,103.47) and (251.9,105.47) .. (249.4,105.47) .. controls (246.9,105.47) and (244.88,103.47) .. (244.88,101) -- cycle ;
\draw  [fill={rgb, 255:red, 255; green, 255; blue, 255 }  ,fill opacity=1 ] (326.66,79.34) .. controls (326.66,76.87) and (328.68,74.87) .. (331.18,74.87) .. controls (333.68,74.87) and (335.7,76.87) .. (335.7,79.34) .. controls (335.7,81.81) and (333.68,83.81) .. (331.18,83.81) .. controls (328.68,83.81) and (326.66,81.81) .. (326.66,79.34) -- cycle ;
\draw  [fill={rgb, 255:red, 255; green, 255; blue, 255 }  ,fill opacity=1 ] (296.88,101) .. controls (296.88,98.53) and (298.9,96.53) .. (301.4,96.53) .. controls (303.9,96.53) and (305.92,98.53) .. (305.92,101) .. controls (305.92,103.47) and (303.9,105.47) .. (301.4,105.47) .. controls (298.9,105.47) and (296.88,103.47) .. (296.88,101) -- cycle ;
\draw  [color={rgb, 255:red, 0; green, 0; blue, 0 }  ,draw opacity=1 ][fill={rgb, 255:red, 255; green, 255; blue, 255 }  ,fill opacity=1 ] (348.72,101.15) .. controls (348.72,98.68) and (350.75,96.67) .. (353.24,96.67) .. controls (355.74,96.67) and (357.77,98.68) .. (357.77,101.15) .. controls (357.77,103.61) and (355.74,105.62) .. (353.24,105.62) .. controls (350.75,105.62) and (348.72,103.61) .. (348.72,101.15) -- cycle ;
\draw  [fill={rgb, 255:red, 255; green, 255; blue, 255 }  ,fill opacity=1 ] (326.95,123.28) .. controls (326.95,120.81) and (328.98,118.81) .. (331.47,118.81) .. controls (333.97,118.81) and (336,120.81) .. (336,123.28) .. controls (336,125.75) and (333.97,127.75) .. (331.47,127.75) .. controls (328.98,127.75) and (326.95,125.75) .. (326.95,123.28) -- cycle ;
\draw  [fill={rgb, 255:red, 255; green, 255; blue, 255 }  ,fill opacity=1 ] (271.68,101) .. controls (271.68,98.53) and (273.7,96.53) .. (276.2,96.53) .. controls (278.7,96.53) and (280.72,98.53) .. (280.72,101) .. controls (280.72,103.47) and (278.7,105.47) .. (276.2,105.47) .. controls (273.7,105.47) and (271.68,103.47) .. (271.68,101) -- cycle ;
\draw  [fill={rgb, 255:red, 0; green, 7; blue, 255 }  ,fill opacity=1 ] (326.47,101.22) .. controls (326.47,98.75) and (328.5,96.75) .. (331,96.75) .. controls (333.49,96.75) and (335.52,98.75) .. (335.52,101.22) .. controls (335.52,103.69) and (333.49,105.69) .. (331,105.69) .. controls (328.5,105.69) and (326.47,103.69) .. (326.47,101.22) -- cycle ;
\draw   (241.2,101) .. controls (241.2,89.95) and (256.87,81) .. (276.2,81) .. controls (295.53,81) and (311.2,89.95) .. (311.2,101) .. controls (311.2,112.05) and (295.53,121) .. (276.2,121) .. controls (256.87,121) and (241.2,112.05) .. (241.2,101) -- cycle ;
\draw   (318.68,79.34) .. controls (318.68,73.37) and (324.28,68.54) .. (331.18,68.54) .. controls (338.08,68.54) and (343.68,73.37) .. (343.68,79.34) .. controls (343.68,85.3) and (338.08,90.14) .. (331.18,90.14) .. controls (324.28,90.14) and (318.68,85.3) .. (318.68,79.34) -- cycle ;
\draw   (340.74,101.15) .. controls (340.74,95.18) and (346.34,90.35) .. (353.24,90.35) .. controls (360.15,90.35) and (365.74,95.18) .. (365.74,101.15) .. controls (365.74,107.11) and (360.15,111.95) .. (353.24,111.95) .. controls (346.34,111.95) and (340.74,107.11) .. (340.74,101.15) -- cycle ;
\draw   (318.97,123.28) .. controls (318.97,117.32) and (324.57,112.48) .. (331.47,112.48) .. controls (338.38,112.48) and (343.97,117.32) .. (343.97,123.28) .. controls (343.97,129.25) and (338.38,134.08) .. (331.47,134.08) .. controls (324.57,134.08) and (318.97,129.25) .. (318.97,123.28) -- cycle ;
\draw    (248.2,125.8) -- (181.82,173.83) ;
\draw [shift={(180.2,175)}, rotate = 324.11] [color={rgb, 255:red, 0; green, 0; blue, 0 }  ][line width=0.75]    (10.93,-3.29) .. controls (6.95,-1.4) and (3.31,-0.3) .. (0,0) .. controls (3.31,0.3) and (6.95,1.4) .. (10.93,3.29)   ;
\draw    (85.4,192.2) -- (110.6,192.2) ;
\draw    (58.6,192.2) -- (85.4,192.2) ;
\draw  [fill={rgb, 255:red, 255; green, 255; blue, 255 }  ,fill opacity=1 ] (135.86,170.54) .. controls (135.86,168.07) and (137.88,166.07) .. (140.38,166.07) .. controls (142.88,166.07) and (144.9,168.07) .. (144.9,170.54) .. controls (144.9,173.01) and (142.88,175.01) .. (140.38,175.01) .. controls (137.88,175.01) and (135.86,173.01) .. (135.86,170.54) -- cycle ;
\draw  [fill={rgb, 255:red, 255; green, 255; blue, 255 }  ,fill opacity=1 ] (106.08,192.2) .. controls (106.08,189.73) and (108.1,187.73) .. (110.6,187.73) .. controls (113.1,187.73) and (115.12,189.73) .. (115.12,192.2) .. controls (115.12,194.67) and (113.1,196.67) .. (110.6,196.67) .. controls (108.1,196.67) and (106.08,194.67) .. (106.08,192.2) -- cycle ;
\draw  [color={rgb, 255:red, 0; green, 0; blue, 0 }  ,draw opacity=1 ][fill={rgb, 255:red, 255; green, 255; blue, 255 }  ,fill opacity=1 ] (157.92,192.35) .. controls (157.92,189.88) and (159.95,187.87) .. (162.44,187.87) .. controls (164.94,187.87) and (166.97,189.88) .. (166.97,192.35) .. controls (166.97,194.81) and (164.94,196.82) .. (162.44,196.82) .. controls (159.95,196.82) and (157.92,194.81) .. (157.92,192.35) -- cycle ;
\draw  [fill={rgb, 255:red, 255; green, 255; blue, 255 }  ,fill opacity=1 ] (136.15,214.48) .. controls (136.15,212.01) and (138.18,210.01) .. (140.67,210.01) .. controls (143.17,210.01) and (145.2,212.01) .. (145.2,214.48) .. controls (145.2,216.95) and (143.17,218.95) .. (140.67,218.95) .. controls (138.18,218.95) and (136.15,216.95) .. (136.15,214.48) -- cycle ;
\draw  [fill={rgb, 255:red, 255; green, 255; blue, 255 }  ,fill opacity=1 ] (80.88,192.2) .. controls (80.88,189.73) and (82.9,187.73) .. (85.4,187.73) .. controls (87.9,187.73) and (89.92,189.73) .. (89.92,192.2) .. controls (89.92,194.67) and (87.9,196.67) .. (85.4,196.67) .. controls (82.9,196.67) and (80.88,194.67) .. (80.88,192.2) -- cycle ;
\draw  [fill={rgb, 255:red, 0; green, 7; blue, 255 }  ,fill opacity=1 ] (135.67,192.42) .. controls (135.67,189.95) and (137.7,187.95) .. (140.2,187.95) .. controls (142.69,187.95) and (144.72,189.95) .. (144.72,192.42) .. controls (144.72,194.89) and (142.69,196.89) .. (140.2,196.89) .. controls (137.7,196.89) and (135.67,194.89) .. (135.67,192.42) -- cycle ;
\draw   (50.4,192.2) .. controls (50.4,181.15) and (66.07,172.2) .. (85.4,172.2) .. controls (104.73,172.2) and (120.4,181.15) .. (120.4,192.2) .. controls (120.4,203.25) and (104.73,212.2) .. (85.4,212.2) .. controls (66.07,212.2) and (50.4,203.25) .. (50.4,192.2) -- cycle ;
\draw  [color={rgb, 255:red, 0; green, 0; blue, 0 }  ,draw opacity=1 ][fill={rgb, 255:red, 255; green, 255; blue, 255 }  ,fill opacity=1 ] (54.08,192.2) .. controls (54.08,189.73) and (56.1,187.73) .. (58.6,187.73) .. controls (61.1,187.73) and (63.12,189.73) .. (63.12,192.2) .. controls (63.12,194.67) and (61.1,196.67) .. (58.6,196.67) .. controls (56.1,196.67) and (54.08,194.67) .. (54.08,192.2) -- cycle ;
\draw    (281,199.4) -- (306.2,199.4) ;
\draw    (306.2,199.4) -- (335.8,199.62) ;
\draw    (335.8,199.62) -- (358.04,199.55) ;
\draw    (335.8,199.62) -- (335.98,177.74) ;
\draw    (336.27,221.68) -- (335.8,199.62) ;
\draw    (254.2,199.4) -- (281,199.4) ;
\draw  [color={rgb, 255:red, 0; green, 0; blue, 0 }  ,draw opacity=1 ][fill={rgb, 255:red, 255; green, 255; blue, 255 }  ,fill opacity=1 ] (249.68,199.4) .. controls (249.68,196.93) and (251.7,194.93) .. (254.2,194.93) .. controls (256.7,194.93) and (258.72,196.93) .. (258.72,199.4) .. controls (258.72,201.87) and (256.7,203.87) .. (254.2,203.87) .. controls (251.7,203.87) and (249.68,201.87) .. (249.68,199.4) -- cycle ;
\draw  [fill={rgb, 255:red, 255; green, 255; blue, 255 }  ,fill opacity=1 ] (331.46,177.74) .. controls (331.46,175.27) and (333.48,173.27) .. (335.98,173.27) .. controls (338.48,173.27) and (340.5,175.27) .. (340.5,177.74) .. controls (340.5,180.21) and (338.48,182.21) .. (335.98,182.21) .. controls (333.48,182.21) and (331.46,180.21) .. (331.46,177.74) -- cycle ;
\draw  [fill={rgb, 255:red, 0; green, 7; blue, 255 }  ,fill opacity=1 ] (301.68,199.4) .. controls (301.68,196.93) and (303.7,194.93) .. (306.2,194.93) .. controls (308.7,194.93) and (310.72,196.93) .. (310.72,199.4) .. controls (310.72,201.87) and (308.7,203.87) .. (306.2,203.87) .. controls (303.7,203.87) and (301.68,201.87) .. (301.68,199.4) -- cycle ;
\draw  [color={rgb, 255:red, 0; green, 0; blue, 0 }  ,draw opacity=1 ][fill={rgb, 255:red, 255; green, 255; blue, 255 }  ,fill opacity=1 ] (353.52,199.55) .. controls (353.52,197.08) and (355.55,195.07) .. (358.04,195.07) .. controls (360.54,195.07) and (362.57,197.08) .. (362.57,199.55) .. controls (362.57,202.01) and (360.54,204.02) .. (358.04,204.02) .. controls (355.55,204.02) and (353.52,202.01) .. (353.52,199.55) -- cycle ;
\draw  [fill={rgb, 255:red, 255; green, 255; blue, 255 }  ,fill opacity=1 ] (331.75,221.68) .. controls (331.75,219.21) and (333.78,217.21) .. (336.27,217.21) .. controls (338.77,217.21) and (340.8,219.21) .. (340.8,221.68) .. controls (340.8,224.15) and (338.77,226.15) .. (336.27,226.15) .. controls (333.78,226.15) and (331.75,224.15) .. (331.75,221.68) -- cycle ;
\draw  [fill={rgb, 255:red, 255; green, 255; blue, 255 }  ,fill opacity=1 ] (276.48,199.4) .. controls (276.48,196.93) and (278.5,194.93) .. (281,194.93) .. controls (283.5,194.93) and (285.52,196.93) .. (285.52,199.4) .. controls (285.52,201.87) and (283.5,203.87) .. (281,203.87) .. controls (278.5,203.87) and (276.48,201.87) .. (276.48,199.4) -- cycle ;
\draw  [fill={rgb, 255:red, 0; green, 7; blue, 255 }  ,fill opacity=1 ] (331.27,199.62) .. controls (331.27,197.15) and (333.3,195.15) .. (335.8,195.15) .. controls (338.29,195.15) and (340.32,197.15) .. (340.32,199.62) .. controls (340.32,202.09) and (338.29,204.09) .. (335.8,204.09) .. controls (333.3,204.09) and (331.27,202.09) .. (331.27,199.62) -- cycle ;
\draw    (190,194.4) -- (224.6,194.59) ;
\draw [shift={(226.6,194.6)}, rotate = 180.31] [color={rgb, 255:red, 0; green, 0; blue, 0 }  ][line width=0.75]    (10.93,-3.29) .. controls (6.95,-1.4) and (3.31,-0.3) .. (0,0) .. controls (3.31,0.3) and (6.95,1.4) .. (10.93,3.29)   ;

\draw (239.2,65.6) node [anchor=north west][inner sep=0.75pt]  [font=\small]  {$W_{1}$};
\draw (320.8,52) node [anchor=north west][inner sep=0.75pt]  [font=\small]  {$W_{2}$};
\draw (368.4,94) node [anchor=north west][inner sep=0.75pt]  [font=\small]  {$W_{3}$};
\draw (320.8,138.8) node [anchor=north west][inner sep=0.75pt]  [font=\small]  {$W_{4}$};
\draw (72.4,226.4) node [anchor=north west][inner sep=0.75pt]  [font=\small]  {$G[ W_{1} \cup B]$};

\end{tikzpicture}

    \caption{An example of one step of the positive semidefinite forcing process}
    \label{fig:PSDforcingprocess}
\end{figure}
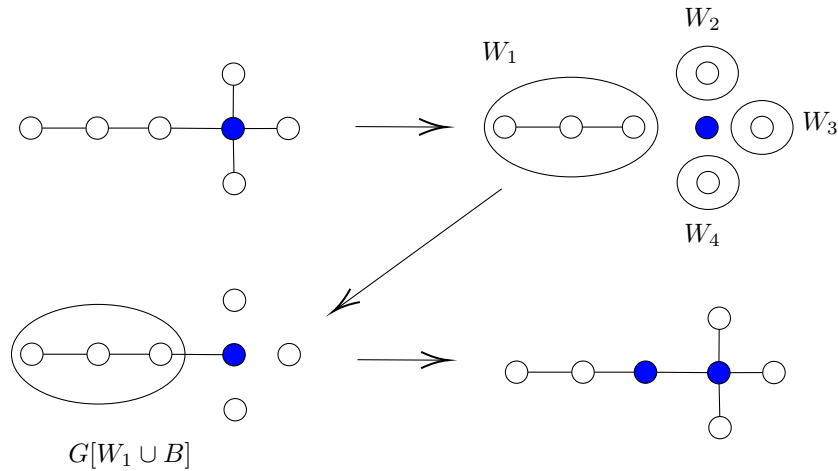

An example of a first step of the positive semidefinite forcing process can be seen in  Figure \ref{fig:PSDforcingprocess}. We note that if there is only one connected component of $G-B^+$, then the positive semidefinite forcing color change rule is the same as standard forcing (this is highlighted in Section \ref{sec:Forts}). Of course, one can color all vertices or $n-1$ vertices of a graph blue and force the entire graph with this forcing process. Thus, the interest lies in determining the fewest number of vertices needed to turn the entire graph blue.

\begin{defn}
   The \emph{positive semidefinite forcing number} of a graph $G$, denoted by $Z^+(G)$, is given by
   $$Z_+(G) = \min \{ |B^+| : B^+\subseteq V(G) \text{ is a positive semidefinite forcing set of } G\}.$$
\end{defn}

 Though positive semidefinite forcing has been well studied, we are interested in the positive semidefinite forcing process when there are leaks present in our graph. The study of the effect of leaks on the zero forcing process on graphs was first introduced in \cite{DillmanKenter}, and has garnered much recent interest \cite{LeakyReslience,GeneralizationsLeaky,D-2LeakyCube}. The following definition of leaks in a graph was introduced in \cite{DillmanKenter}. 

\begin{defn}
    For any graph $G$, a \emph{leak} in $G$ is any vertex $v\in V(G)$ that loses its ability to force. 
\end{defn}

Due to the positive semidefinite forcing process, a leak $v\in V(G)$ in a graph may be visualized by imagining an additional vertex $u$ adjacent only to $v$ which hinders the forcing process since the vertex $v$ would have an additional non-blue vertex. 


With the addition of leaks in a graph, we now define a $\ell$-leaky positive semidefinite forcing set and the $\ell$-leaky positive semidefinite forcing number of a graph, which includes the possibility of having no leaks. 

\begin{defn}
    Let $B_{(\ell)}^+\subseteq V(G)$ be a subset of vertices initial colored blue. After choosing this set of initially colored vertices, choose $\ell$ vertices to be leaks in $G$. Then, $B^+_{(\ell)}$ is an \emph{$\ell$-leaky positive semidefinite forcing set} if it can still force $G$ following the positive semidefinite color change rule, despite the presence of $\ell$ leaks, wherever the leaks are placed in $G$.
\end{defn}

\begin{defn}
    For any graph $G$ and $\ell\in \mathbb{Z}_{ \geq 0}$, the \emph{$\ell$-leaky positive semidefinite forcing number}, denoted $Z^+_{(\ell)}(G)$, is given by
    $$Z^+_{(\ell)}(G) = \min \{ |B_{(\ell)}^+| : B_{(\ell)}^+\subseteq V(G) \text{ is an } \ell \text{-leaky positive semidefinite forcing set}\}.$$
\end{defn}

When $\ell = 0$, notice $\psdl{0}{G} = Z_+(G)$. We note that the notation used for $\ell$-leaky positive semidefinite forcing shifts the $+$ to a superscript because of the incorporation of $(\ell)$ in the subscript. See Figure \ref{fig:PSDLeakyExample} for an example of a graph and $\ell$-leaky sets for various $\ell$ which highlight the difference between forcing sets when adding leaks to a graph $G$.

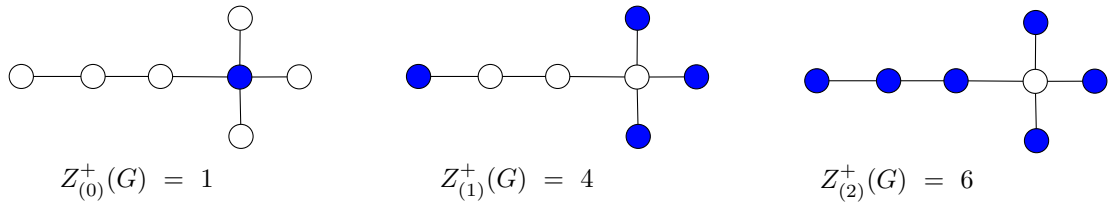
\begin{figure}[h!]
    \centering
    \begin{tikzpicture}[x=0.75pt,y=0.75pt,yscale=-1.35,xscale=1.35]

\draw    (215,184.8) -- (240.2,184.8) ;
\draw    (240.2,184.8) -- (269.8,185.02) ;
\draw    (269.8,185.02) -- (292.04,184.95) ;
\draw    (269.8,185.02) -- (269.98,163.14) ;
\draw    (270.27,207.08) -- (269.8,185.02) ;
\draw    (188.2,184.8) -- (215,184.8) ;
\draw  [color={rgb, 255:red, 0; green, 0; blue, 0 }  ,draw opacity=1 ][fill={rgb, 255:red, 0; green, 7; blue, 255 }  ,fill opacity=1 ] (183.68,184.8) .. controls (183.68,182.33) and (185.7,180.33) .. (188.2,180.33) .. controls (190.7,180.33) and (192.72,182.33) .. (192.72,184.8) .. controls (192.72,187.27) and (190.7,189.27) .. (188.2,189.27) .. controls (185.7,189.27) and (183.68,187.27) .. (183.68,184.8) -- cycle ;
\draw  [fill={rgb, 255:red, 0; green, 7; blue, 255 }  ,fill opacity=1 ] (265.46,163.14) .. controls (265.46,160.67) and (267.48,158.67) .. (269.98,158.67) .. controls (272.48,158.67) and (274.5,160.67) .. (274.5,163.14) .. controls (274.5,165.61) and (272.48,167.61) .. (269.98,167.61) .. controls (267.48,167.61) and (265.46,165.61) .. (265.46,163.14) -- cycle ;
\draw  [fill={rgb, 255:red, 255; green, 255; blue, 255 }  ,fill opacity=1 ] (235.68,184.8) .. controls (235.68,182.33) and (237.7,180.33) .. (240.2,180.33) .. controls (242.7,180.33) and (244.72,182.33) .. (244.72,184.8) .. controls (244.72,187.27) and (242.7,189.27) .. (240.2,189.27) .. controls (237.7,189.27) and (235.68,187.27) .. (235.68,184.8) -- cycle ;
\draw  [color={rgb, 255:red, 0; green, 0; blue, 0 }  ,draw opacity=1 ][fill={rgb, 255:red, 0; green, 7; blue, 255 }  ,fill opacity=1 ] (287.52,184.95) .. controls (287.52,182.48) and (289.55,180.47) .. (292.04,180.47) .. controls (294.54,180.47) and (296.57,182.48) .. (296.57,184.95) .. controls (296.57,187.41) and (294.54,189.42) .. (292.04,189.42) .. controls (289.55,189.42) and (287.52,187.41) .. (287.52,184.95) -- cycle ;
\draw  [fill={rgb, 255:red, 0; green, 7; blue, 255 }  ,fill opacity=1 ] (265.75,207.08) .. controls (265.75,204.61) and (267.78,202.61) .. (270.27,202.61) .. controls (272.77,202.61) and (274.8,204.61) .. (274.8,207.08) .. controls (274.8,209.55) and (272.77,211.55) .. (270.27,211.55) .. controls (267.78,211.55) and (265.75,209.55) .. (265.75,207.08) -- cycle ;
\draw  [fill={rgb, 255:red, 255; green, 255; blue, 255 }  ,fill opacity=1 ] (210.48,184.8) .. controls (210.48,182.33) and (212.5,180.33) .. (215,180.33) .. controls (217.5,180.33) and (219.52,182.33) .. (219.52,184.8) .. controls (219.52,187.27) and (217.5,189.27) .. (215,189.27) .. controls (212.5,189.27) and (210.48,187.27) .. (210.48,184.8) -- cycle ;
\draw  [fill={rgb, 255:red, 255; green, 255; blue, 255 }  ,fill opacity=1 ] (265.27,185.02) .. controls (265.27,182.55) and (267.3,180.55) .. (269.8,180.55) .. controls (272.29,180.55) and (274.32,182.55) .. (274.32,185.02) .. controls (274.32,187.49) and (272.29,189.49) .. (269.8,189.49) .. controls (267.3,189.49) and (265.27,187.49) .. (265.27,185.02) -- cycle ;
\draw    (66.6,184.8) -- (91.8,184.8) ;
\draw    (91.8,184.8) -- (121.4,185.02) ;
\draw    (121.4,185.02) -- (143.64,184.95) ;
\draw    (121.4,185.02) -- (121.58,163.14) ;
\draw    (121.87,207.08) -- (121.4,185.02) ;
\draw    (39.8,184.8) -- (66.6,184.8) ;
\draw  [color={rgb, 255:red, 0; green, 0; blue, 0 }  ,draw opacity=1 ][fill={rgb, 255:red, 255; green, 255; blue, 255 }  ,fill opacity=1 ] (35.28,184.8) .. controls (35.28,182.33) and (37.3,180.33) .. (39.8,180.33) .. controls (42.3,180.33) and (44.32,182.33) .. (44.32,184.8) .. controls (44.32,187.27) and (42.3,189.27) .. (39.8,189.27) .. controls (37.3,189.27) and (35.28,187.27) .. (35.28,184.8) -- cycle ;
\draw  [fill={rgb, 255:red, 255; green, 255; blue, 255 }  ,fill opacity=1 ] (117.06,163.14) .. controls (117.06,160.67) and (119.08,158.67) .. (121.58,158.67) .. controls (124.08,158.67) and (126.1,160.67) .. (126.1,163.14) .. controls (126.1,165.61) and (124.08,167.61) .. (121.58,167.61) .. controls (119.08,167.61) and (117.06,165.61) .. (117.06,163.14) -- cycle ;
\draw  [fill={rgb, 255:red, 255; green, 255; blue, 255 }  ,fill opacity=1 ] (87.28,184.8) .. controls (87.28,182.33) and (89.3,180.33) .. (91.8,180.33) .. controls (94.3,180.33) and (96.32,182.33) .. (96.32,184.8) .. controls (96.32,187.27) and (94.3,189.27) .. (91.8,189.27) .. controls (89.3,189.27) and (87.28,187.27) .. (87.28,184.8) -- cycle ;
\draw  [color={rgb, 255:red, 0; green, 0; blue, 0 }  ,draw opacity=1 ][fill={rgb, 255:red, 255; green, 255; blue, 255 }  ,fill opacity=1 ] (139.12,184.95) .. controls (139.12,182.48) and (141.15,180.47) .. (143.64,180.47) .. controls (146.14,180.47) and (148.17,182.48) .. (148.17,184.95) .. controls (148.17,187.41) and (146.14,189.42) .. (143.64,189.42) .. controls (141.15,189.42) and (139.12,187.41) .. (139.12,184.95) -- cycle ;
\draw  [fill={rgb, 255:red, 255; green, 255; blue, 255 }  ,fill opacity=1 ] (117.35,207.08) .. controls (117.35,204.61) and (119.38,202.61) .. (121.87,202.61) .. controls (124.37,202.61) and (126.4,204.61) .. (126.4,207.08) .. controls (126.4,209.55) and (124.37,211.55) .. (121.87,211.55) .. controls (119.38,211.55) and (117.35,209.55) .. (117.35,207.08) -- cycle ;
\draw  [fill={rgb, 255:red, 255; green, 255; blue, 255 }  ,fill opacity=1 ] (62.08,184.8) .. controls (62.08,182.33) and (64.1,180.33) .. (66.6,180.33) .. controls (69.1,180.33) and (71.12,182.33) .. (71.12,184.8) .. controls (71.12,187.27) and (69.1,189.27) .. (66.6,189.27) .. controls (64.1,189.27) and (62.08,187.27) .. (62.08,184.8) -- cycle ;
\draw  [fill={rgb, 255:red, 0; green, 7; blue, 255 }  ,fill opacity=1 ] (116.87,185.02) .. controls (116.87,182.55) and (118.9,180.55) .. (121.4,180.55) .. controls (123.89,180.55) and (125.92,182.55) .. (125.92,185.02) .. controls (125.92,187.49) and (123.89,189.49) .. (121.4,189.49) .. controls (118.9,189.49) and (116.87,187.49) .. (116.87,185.02) -- cycle ;
\draw    (363.8,186.4) -- (389,186.4) ;
\draw    (389,186.4) -- (418.6,186.62) ;
\draw    (418.6,186.62) -- (440.84,186.55) ;
\draw    (418.6,186.62) -- (418.78,164.74) ;
\draw    (419.07,208.68) -- (418.6,186.62) ;
\draw    (337,186.4) -- (363.8,186.4) ;
\draw  [color={rgb, 255:red, 0; green, 0; blue, 0 }  ,draw opacity=1 ][fill={rgb, 255:red, 0; green, 7; blue, 255 }  ,fill opacity=1 ] (332.48,186.4) .. controls (332.48,183.93) and (334.5,181.93) .. (337,181.93) .. controls (339.5,181.93) and (341.52,183.93) .. (341.52,186.4) .. controls (341.52,188.87) and (339.5,190.87) .. (337,190.87) .. controls (334.5,190.87) and (332.48,188.87) .. (332.48,186.4) -- cycle ;
\draw  [fill={rgb, 255:red, 0; green, 7; blue, 255 }  ,fill opacity=1 ] (414.26,164.74) .. controls (414.26,162.27) and (416.28,160.27) .. (418.78,160.27) .. controls (421.28,160.27) and (423.3,162.27) .. (423.3,164.74) .. controls (423.3,167.21) and (421.28,169.21) .. (418.78,169.21) .. controls (416.28,169.21) and (414.26,167.21) .. (414.26,164.74) -- cycle ;
\draw  [fill={rgb, 255:red, 0; green, 7; blue, 255 }  ,fill opacity=1 ] (384.48,186.4) .. controls (384.48,183.93) and (386.5,181.93) .. (389,181.93) .. controls (391.5,181.93) and (393.52,183.93) .. (393.52,186.4) .. controls (393.52,188.87) and (391.5,190.87) .. (389,190.87) .. controls (386.5,190.87) and (384.48,188.87) .. (384.48,186.4) -- cycle ;
\draw  [color={rgb, 255:red, 0; green, 0; blue, 0 }  ,draw opacity=1 ][fill={rgb, 255:red, 0; green, 7; blue, 255 }  ,fill opacity=1 ] (436.32,186.55) .. controls (436.32,184.08) and (438.35,182.07) .. (440.84,182.07) .. controls (443.34,182.07) and (445.37,184.08) .. (445.37,186.55) .. controls (445.37,189.01) and (443.34,191.02) .. (440.84,191.02) .. controls (438.35,191.02) and (436.32,189.01) .. (436.32,186.55) -- cycle ;
\draw  [fill={rgb, 255:red, 0; green, 7; blue, 255 }  ,fill opacity=1 ] (414.55,208.68) .. controls (414.55,206.21) and (416.58,204.21) .. (419.07,204.21) .. controls (421.57,204.21) and (423.6,206.21) .. (423.6,208.68) .. controls (423.6,211.15) and (421.57,213.15) .. (419.07,213.15) .. controls (416.58,213.15) and (414.55,211.15) .. (414.55,208.68) -- cycle ;
\draw  [fill={rgb, 255:red, 0; green, 7; blue, 255 }  ,fill opacity=1 ] (359.28,186.4) .. controls (359.28,183.93) and (361.3,181.93) .. (363.8,181.93) .. controls (366.3,181.93) and (368.32,183.93) .. (368.32,186.4) .. controls (368.32,188.87) and (366.3,190.87) .. (363.8,190.87) .. controls (361.3,190.87) and (359.28,188.87) .. (359.28,186.4) -- cycle ;
\draw  [fill={rgb, 255:red, 255; green, 255; blue, 255 }  ,fill opacity=1 ] (414.07,186.62) .. controls (414.07,184.15) and (416.1,182.15) .. (418.6,182.15) .. controls (421.09,182.15) and (423.12,184.15) .. (423.12,186.62) .. controls (423.12,189.09) and (421.09,191.09) .. (418.6,191.09) .. controls (416.1,191.09) and (414.07,189.09) .. (414.07,186.62) -- cycle ;

\draw (53.2,216.4) node [anchor=north west][inner sep=0.75pt]  [font=\small]  {$Z_{( 0)}^{+}( G) \ =\ 1$};
\draw (195.2,216.4) node [anchor=north west][inner sep=0.75pt]  [font=\small]  {$Z_{( 1)}^{+}( G) \ =\ 4$};
\draw (337.2,216.8) node [anchor=north west][inner sep=0.75pt]  [font=\small]  {$Z_{( 2)}^{+}( G) \ =\ 6$};

\end{tikzpicture}

    \caption{Examples of a graph and $\ell$-leaky positive semidefinite forcing sets for each.}
    \label{fig:PSDLeakyExample}
\end{figure}

\section{Properties of leaky positive semidefinite forcing}\label{sec:PropertiesLeakyForcing}

\quad \quad In this section, crucial theorems about positive semidefinite leaky forcing are established, which are useful in many results that follow in subsequent sections. As noted previously, for any graph $G$ and initial set of vertices colored blue $B^+\subseteq V(G)$, if $G-B^+$ is a single component, the positive semidefinite color change rule is the same as the standard zero forcing rule. Also, any valid $\ell$-leaky standard force on a graph is a valid $\ell$-leaky positive semidefinite force on the same graph. The following proposition, originally stated for $\ell=0$ in \cite{PSDOrigins}, follows from this sentiment.

\begin{prop}
\label{prop:psdlleqleaky}
    For any graph $G$ of order $n$ with $\ell$ leaks, $$\psdl{\ell}{G}\leq Z_{(\ell)}(G).$$
\end{prop}

This bound was shown to be tight for $\ell = 0$ in \cite{PSDOrigins} and the graph families in Section \ref{sec:PathCycleComplete}, together with the results in \cite{DillmanKenter}, show that this bound is tight for all $0\le \ell \le n$. Leaks cause certain vertices to lose their forcing ability, disrupting the forcing process. This is the key to the following theorem and is crucial to establishing the results in Section \ref{sec:GraphFamilies}. We note that the following is true regardless of what forcing variation is used and is stated here specifically for $\ell$-leaky positive semidefinite zero forcing. 


\begin{thm}
\label{thm:Inequalities}
    For any graph $G$ of order $n$, 
    $$\psdl{0}{G}\leq\psdl{1}{G}\leq\psdl{2}{G}\leq\cdots\leq\psdl{n}{G}.$$
\end{thm}

\begin{proof}
     Since leaks disrupt the forcing process, a set that can force an entire graph in the presence of any $\ell$ leaks is sufficient to force a graph with less than $\ell$ leaks, noting that if removing a leak created any change to the coloring process, it would be true that without that removed leak, a vertex that before was unable to force is now able to force.
\end{proof}

\begin{thm}
\label{thm:degree}
    For any graph $G$ with $\ell$ leaks, any $\ell$-leaky positive semidefinite forcing set  includes all vertices of degree $\ell$ or less.
\end{thm}

\begin{proof}
   Let $B^+_{(\ell)}$ be an $\ell$-leaky positive semidefinite forcing set of a graph $G$ with $\ell$ leaks. Suppose, to the contrary, that not every vertex of degree $\ell$ or less is initially colored blue. So, there is at least one $v\in V(G)\setminus B^+_{(\ell)}$ such that $|N_G(v)| \le \ell$. Assume those vertices in $u\in N_G(v)$ are leaks in $G$. Now, regardless of which $G[W_{i} \cup B^+_{(\ell)}]$ is observed to attempt to force $v$, none of the neighbors of $v$ would be able to force, and therefore $B^+_{(\ell)}$ could not be an $\ell$-leaky positive semidefinite forcing set. Therefore, the $\ell$-leaky positive semidefinite forcing set must contain all those vertices in G with degree $\ell$ or less.
\end{proof}

\begin{cor}\label{cor:mindegree}
    For any graph $G$ on $n$ vertices, $Z_{(\ell)}^+(G) = n$ if and only if $\Delta(G) \le \ell$. 
\end{cor}


We now introduce the notion of a set of forcing processes, introduced in \cite{LeakyReslience}, and the set of all possible forces to prove how the $\ell$-leaky forcing changes when we remove an edge from the graph $G$. 

\begin{defn}
   Let $B^+\subseteq V(G)$ be a subset of vertices initially colored blue. A \textit{set of psd forces}, $\textbf{F}_+$, with respect to the set $B^+$ in $G$, is a set of forces such that there is a chronological list of forces in $\textbf{F}_+$ where each force can occur, and the whole graph becomes blue.
\end{defn}

\begin{defn}
    For any subset of vertices $B^+\subseteq V(G)$, the \textit{set of all possible positive semidefinite forces starting with $B^+$} is denoted $\mathcal{F}^+(B)$. This means that a psd force $u\to v\in \mathcal{F}^+(B)$ exists if there is a set of forces $\textbf{F}_+$ in $G$ starting with the set $B$ that contains $u\to v$.
\end{defn}



Note that the previous definitions extend naturally to $\ell$ leaks appearing in $G$. Thus, when we allow $\ell$ leaks to appear in the graph, the next proposition tells us the number of ways we must have to force a vertex in $\mathcal{F}^+(B_{(\ell)}^+)$ starting with a $\ell$-leaky positive semidefinite forcing set $B_{(\ell)}^+$. An analogous result appeared with regard to $\ell$-leaky standard forcing in \cite{LeakyReslience} and an identical proof to that which follows here.

\begin{prop}\label{prop:EllDiffFForces}
    If $B_{(\ell)}^+$ is an $\ell$-leaky positive semidefinite forcing set, then for all $v\in V(G)\setminus B_{(\ell)}^+$, there exist forces $\ell+1$ distinct ways to force $v$, i.e. $x_1\to v, x_2\to v, ... , x_{\ell+1}\to v \in \mathcal{F}^+(B_{(\ell)}^+)$, $x_i\ne x_j$ for all $i\ne j$. 
\end{prop}

Investigating leaks appearing in graphs further, we determine what happens to the positive semidefinite $\ell$-leaky forcing number when we delete an edge from the graph. Adding the leaks to the graph presents difficulties when deleting an edge, since the leak can appear on those vertices incident to the deleted edge. When $\ell = 0$, the following result was proven for positive semidefinite forcing in \cite{PSDforcingEdgeDelete}.

\begin{thm}\cite{PSDforcingEdgeDelete}
    Let $G$ be a graph. For any edge $e\in E(G)$
        $$-1 \le Z_+(G) - Z_+(G-e) \le 1.$$
\end{thm}

Towards a similar result for one leak, results are needed which give us a way to start with a forcing process for a forcing set and switch to a new, though related, forcing process. The following notation and results are adapted from \cite{LeakyReslience} where authors proved results for leaky standard forcing. We omit proofs that were found to be identical to that in \cite{LeakyReslience} and include ones that needed a varied approach because of the positive semidefinite forcing process. The observed similarities between leaky standard and positive semidefinite forcing again motivated the investigation into forts in Section \ref{sec:Forts}

\begin{defn}
     Given an initial set of blue vertices $B^+$ and set of psd forces $\textbf{F}_+$, we say that the set $(B^+)'$, $B^+\subseteq (B^+)'  \subseteq V(G)$, is \textit{obtained from $B^+$ using $\textbf{F}_+$} if $B^+$ can color $(B^+)'$ using only a subset of $\textbf{F}_+$.
\end{defn}

Let $B$ be a fixed blue vertex set of a graph $G$ with sets of psd forces $\textbf{F}_+$ and $\textbf{F}_+'$. We are able to use the forcing process $\textbf{F}_+$ to obtain a blue vertex set $B'$ from $B$, then continue the forcing with process $\textbf{F}_+'$. Formalizing this, suppose $S\subset V(G)$ and let 
    $$\textbf{F}_+(S) = \{ x\to y \in F_+ : y\notin S\}$$
and 
    $$\textbf{F}_+\setminus \textbf{F}_+(S) = \{x\to y \in F : y\in S\}$$
Note that though the positive semidefinite forcing process can be viewed as induced trees of the graph, the above sets $\textbf{F}_+$ and $\textbf{F}_+'$ consists of a list of single forces $x\to y$. The proof of the following lemma is identical to that for standard forcing in \cite{LeakyReslience}.
\begin{lem}
    Let $B^+$ be a positive semidefinite forcing set of a graph $G$ with sets of psd forces $\textbf{F}_+$ and $\textbf{F}_+'$. Then, the set 
        $$(\textbf{F}_+ \setminus \textbf{F}_+((B^+)') \cup \textbf{F}_+'((B^+)')$$ is a set of psd forces of $B^+$ for any $(B^+)'$ obtained from $B^+$ using $\textbf{F}_+$. 
\end{lem}

From the previous lemma, a result can be proved giving us criteria for a specified set to be a $1$-leaky positive semidefinite forcing set.
\begin{thm}\label{thm:1leakySets}
A set $B$ is a $1$-positive semidefinite leaky forcing set if and only if for all $v\in V(G)\setminus B$, there exists $x\to v$, $y\to v \in \mathcal{F}^+(B)$ with $y\ne x$. 
\end{thm}

\begin{proof}
    The forward direction follows from Proposition \ref{prop:EllDiffFForces}. Now, suppose that $B\subset V(G)$ is a subset of initial blue vertices, and that for all $v\in V(G)\setminus B$, there exists $x\to v,~y\to v\in \mathcal{F}^+(B)$ with $y\ne x$. Let $u_\ell$ be a leaked vertex in $G$ and $\textbf{F}_+$ be some set of psd forces with respect to the set $B$ in $G$. If $u_\ell$ is the end of a forcing tree of $\textbf{F}_+$, then $B$ can force all of $G$ since no force originates from the leaked vertex $u_\ell$. 

    Now assume that $u_\ell$ is not the end of a forcing tree. So, there is some set of vertices $\{z_1, ... , z_j\}$ such that $u_\ell\to z_1, ~ u_\ell \to z_2,~ ... , ~u_\ell\to z_j$. We want to construct a set of forces of $G$ that do not use a force $u_\ell \to v$ for any $v\in V(G)\setminus B$. Let $B'$ be the set of blue vertices obtained from $B$ using $\textbf{F}_+$ such that $u_\ell\to z_1, ~ u_\ell \to z_2,~ ... , ~u_\ell\to z_j$ are valid forces, but $z_i\notin B'$ for all $1\le i \le j$. By assumption, there must exist a collection of vertices $y_1, ... , y_j$ such that $y_i \to z_i ~ \text{ for all } 1\le i \le j$ noting that $y_i \ne u_\ell$ for all $i$, and that not all the $y_i$ need be distinct. 
    
    Since $\mathcal{F}^+(B)$ is the set of all forcing processes starting with $B$, there exists a set of forces $\textbf{F}_+'$ such that $y_i \to z_i \in \textbf{F}_+'$ for all $i$. Since $y_i\to z_i \in \textbf{F}_+'$, we know that $\textbf{F}_+'$ can force $N[y_i]\setminus \{z_i\}$ without vertices in $N[z_i]\setminus \{y_i\}$. Namely, $\textbf{F}_+'$ can force $N[y_i]\setminus \{z_i\}$ without the vertex $u_\ell$. Therefore, $u_\ell\to w \notin \textbf{F}_+'(B')$ for any $w \in V(G)$ and since since the set of forces $\textbf{F}_+ \setminus \textbf{F}_+(B')$ is constructed by removing those forces $u_\ell\to u$ such that $u\in B'$, we get that
        $$(\textbf{F}_+ \setminus \textbf{F}_+(B') \cup \textbf{F}_+'(B')$$
    is a set of forcing processes that do not use forces originating at the leaked vertex $u_\ell$.
\end{proof}

Using Theorem \ref{thm:1leakySets}, we obtain the following result, using an identical proof to that in \cite{LeakyReslience} for $1$-leaky standard forcing, though a proof for an upper bound on the difference eluded us. 

\begin{prop}\label{thm:EdgeDeletion}
    For every graph $G$ and every edge $e\in E(G)$, 
        $$-2 \le Z^+_{(1)}(G) - Z^+_{(1)}(G - e).$$
\end{prop}

 We computationally verified (via SAGE) that the difference in Proposition \ref{thm:EdgeDeletion} is bounded above by one for all connected graphs up to order eight. We provide examples of graphs satisfying the difference $Z^+_{(1)}(G) - Z^+_{(1)}(G - e)$ being one in Figure \ref{fig:DifferenceDeleteEdge} with $1$-leaky sets given for each graph before and after edge deletion. The complete list of graphs given as graph6 strings up to order eight with the difference $Z^+_{(1)}(G) - Z^+_{(1)}(G - e) = 1$ for some edge in the graph $G$ can be found via this url \href{https://drive.google.com/file/d/17r9E-OKHVcZNjVg8iXk0N_gqc2yhXXBe/view?usp=drive_link}{https://tinyurl.com/PSDLeakyForcing}, and because of this we conjecture the following. 

\begin{conj}\label{conj:edgedeletion}
For every graph $G$ and every edge $e\in E(G)$
$$Z^+_{(1)}(G) - Z^+_{(1)}(G - e) \le 1.$$
\end{conj}

\begin{figure}
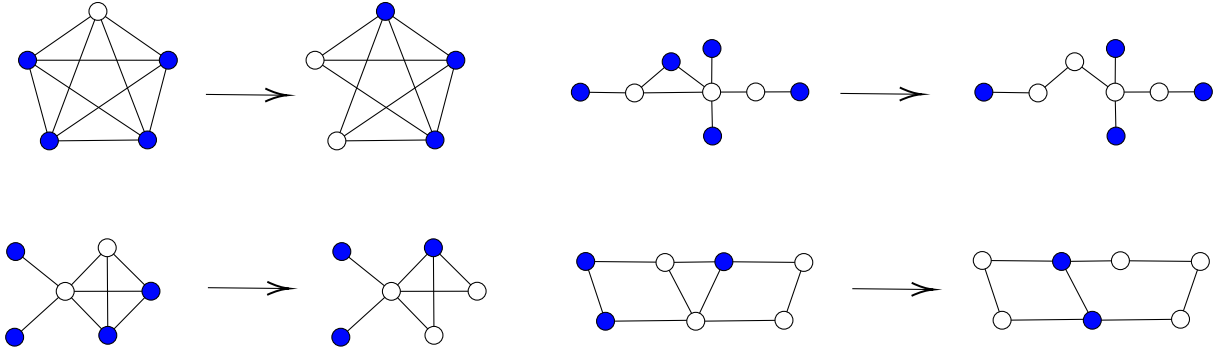

    \centering


    \caption{Examples of graphs satisfying equality in Conjecture \ref{conj:edgedeletion}}
    \label{fig:DifferenceDeleteEdge}
\end{figure}


\section{Particular graph families}\label{sec:GraphFamilies}

\quad \quad Towards understanding $\ell$-leaky positive semidefinite forcing, we determined the $\ell$-leaky positive semidefinite forcing number for various graph families. In particular, we determined this graph parameter for path, cycle, complete, wheel, complete bipartite, and tree graphs. When $\ell = 0$, the positive semidefinite leaky forcing number can be found in \cite{IEPGZF} and references therein. 

\subsection{Path, Cycle, and Complete Graphs}\label{sec:PathCycleComplete}
\begin{thm}
    Let $P_n$ denote the path graph on $n$ vertices. Then,

\begin{align*}
    \psdl{\ell}{P_n}&=\begin{cases}
        1 & \ell = 0,\\
        2 & \ell = 1,\\
        n & \ell \geq 2.
    \end{cases}
\end{align*}
\end{thm}

\begin{proof}
      Suppose $\ell=1$. By Theorem \ref{thm:degree}, all vertices with a degree of one must be colored, which would be both endpoints of $P_n$, implying $2\le \psdl{1}{P_n}$. Equality can be seen with the set of endpoints of any path graph forming a $1$-positive semidefinite leaky forcing set. For $\ell\geq 2$, we have $Z^+_{(\ell)}(G) = n$ by Corollary \ref{cor:mindegree}. 
\end{proof}

\begin{thm}
    Let $C_n$ denote the cycle graph on $n$ vertices. Then,

\begin{align*}
    \psdl{\ell}{C_n}&=\begin{cases}
       2 & \ell=0,1, \\
       n & \ell\geq 2.
    \end{cases}
\end{align*}
\end{thm}

\begin{proof}
      Suppose $\ell=1$. It can be seen that any set of two vertices will form a $1$-leaky positive semidefinite forcing set, and thus $Z^+_{(1)}(C_n) = 2$ by Theorem \ref{thm:Inequalities}. For $\ell\geq 2$, we get $Z^+_{(\ell)}(C_n) = n $ by Corollary \ref{cor:mindegree}. 
\end{proof}

\begin{thm}
    Let $K_n$ denote the complete graph on $n$ vertices. Then,

\begin{align*}
    \psdl{\ell}{K_n}&=\begin{cases}
        n-1 & \ell\leq n-2, \\
        n & \ell\geq n-1.
    \end{cases}
\end{align*}

\end{thm}

\begin{proof}
    
    Consider $\ell$ leaks for $\ell\leq n-2$ and a set $B^+\subset V(G)$ such that $B^+$ is a positive semidefinite forcing set of $K_n$. Since $|B^+| = n-1$ and all vertices $v\in V(K_n)$ have degree $n-1$, it is still possible to force the single uncolored vertex with at most $n-2$ leaks. Therefore, with Theorem \ref{thm:Inequalities}, we have $Z_{(\ell)}^+(K_n) = n-1$ for all $\ell \le n-2$. By Corollary \ref{cor:mindegree}, when $\ell\geq n-1$, we get $Z_{(\ell)}^+(K_n) = n$.
\end{proof}

\subsection{Wheel Graphs}

\begin{thm}
    Let $W_n$ denote the wheel graph on $n+1$ vertices. Then,
    \begin{align*}
    \psdl{\ell}{W_n}&=\begin{cases}
        3 & \ell=0,1, \\
        \lceil \frac{n}{2} \rceil +1 & \ell=2, \\
        n & 2<\ell<n, \\
        n+1 & \ell=n,n+1.
    \end{cases}
\end{align*}
\end{thm}

\begin{proof}

    Suppose $\ell=1$. Since $\psdl{0}{W_n}=3$ and the $1$-leaky standard forcing number is $\leaky{1}{W_n}=3$ \cite{DillmanKenter}, by Proposition \ref{prop:psdlleqleaky} and Theorem \ref{thm:Inequalities}, we have $\psdl{1}{W_n}=3$.
    
    When $2< \ell<n$, by Corollary \ref{cor:mindegree}, any positive semidefinite $\ell$-leaky forcing set must contain all outer vertices on the wheel, since they all have degree $3$. Furthermore, all outer vertices can be shown to be an $\ell$-leaky positive semidefinite leaky set. Therefore, $\psdl{\ell}{W_n} = n$. In the cases of $\ell=n,n+1$, again by Corollary \ref{cor:mindegree}, all vertices on the wheel must be included and $\psdl{\ell}{W_n} = n+1$.

    Finally, in the case of $\ell=2$, we note that as long as the center vertex is initially colored blue and the outer vertices are colored in an alternating pattern so that there are no two adjacent white vertices as seen in Figure \ref{fig:wheels}, the two leaks will be able to appear anywhere without blocking the complete forcing of the graph. So, $\psdl{2}{W_n}\le \lceil \frac{n}{2} \rceil +1$. Also, by Proposition \ref{prop:EllDiffFForces}, removing any vertex from this coloring pattern will result in a vertex having fewer than $3$ possible ways to force it and thus $\psdl{2}{W_n} = \lceil \frac{n}{2} \rceil +1$.
\end{proof}

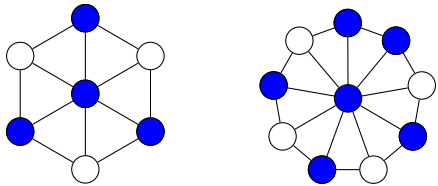
\begin{figure}[h!]
\begin{center}
        \begin{tikzpicture}
            [scale=.75, nodes={circle, draw}]
            \graph[clockwise, radius=1.0cm, empty nodes] {subgraph C_n [n=6,name=A] };
            \node at  ($(A 1)!.5!(A 4)$) (C){};
            \draw (A 1) -- (C);
            \draw (A 2) -- (C);
            \draw (A 3) -- (C);
            \draw (A 4) -- (C);
            \draw (A 5) -- (C);
            \draw (A 6) -- (C);
            \node[shape=circle, draw=black, fill=blue] at (A 1){};
            \node[shape=circle, draw=black, fill=blue] at (A 3){};
            \node[shape=circle, draw=black, fill=blue] at (A 5){};
            \node[shape=circle, draw=black, fill=blue] at (C){};
        \end{tikzpicture}
        \hspace{1cm}
        \begin{tikzpicture}
            [scale=0.75, nodes={circle, draw}]
            \graph[clockwise, radius=1.0cm, empty nodes] {subgraph C_n [n=9,name=A] };
            \node[shape=circle, draw=black, fill=blue] at  ($(A 1)!.5!(A 4)!.33!(A 7)$) (C){};
            \draw (A 1) -- (C);
            \draw (A 2) -- (C);
            \draw (A 3) -- (C);
            \draw (A 4) -- (C);
            \draw (A 5) -- (C);
            \draw (A 6) -- (C);
            \draw (A 7) -- (C);
            \draw (A 8) -- (C);
            \draw (A 9) -- (C);
            \node[shape=circle, draw=black, fill=blue] at (A 1){};
            \node[shape=circle, draw=black, fill=blue] at (A 2){};
            \node[shape=circle, draw=black, fill=blue] at (A 4){};
            \node[shape=circle, draw=black, fill=blue] at (A 6){};
            \node[shape=circle, draw=black, fill=blue] at (A 8){};
        \end{tikzpicture}
\end{center}
 \caption{Coloring pattern for wheel graphs with two leaks}
    \label{fig:wheels}
\end{figure}

\subsection{Complete Bipartite Graphs}
Let $K_{m,n}$ be the complete bipartite graph, with vertex set $\{v_1, v_2, ... , v_n, v_{n+1}, ... , v_{n+m}\}$ where the vertex subsets $V_1 = \{v_1, ... , v_n\}$ forms one independent set and $V_2 = V(G)\setminus V_1$ forms the second independent set of $K_{n,m}$.
\begin{thm}
For any $n,m\in\mathbb{Z}_+$, 

\begin{align*}
\label{psdlkmn}
    \psdl{\ell}{K_{m,n}}&=\begin{cases}
        \mathrm{min} (m,n) & \ell < \mathrm{min} (m,n),\\
        \mathrm{max} (m,n) & \mathrm{min}(m,n) \leq \ell < \mathrm{max}(m,n), \\
        m + n &  \mathrm{max}(m,n) \le \ell.
        \end{cases}
\end{align*}
\end{thm}

\begin{proof}
     Assuming that $n\leq m$, when $\ell < \mathrm{min} (m,n)$, we are able to color the independent set $V_1$ and force the graph no matter where the leaks are placed. Therefore, using Theorem \ref{thm:Inequalities}, $\psdl{\ell}{K_{m,n}} = \min\{m,n\}$.
    
    When $\mathrm{min}(m,n) \leq \ell < \mathrm{max}(m,n)$, by Theorem \ref{thm:degree}, our initial blue set must contain the vertices in $V_2 = V(G)\setminus V_1$. Furthermore, this set can be seen to be a $\ell$-leaky positive semidefinite forcing set, and thus $\psdl{\ell}{K_{m,n}} = \mathrm{max}(m,n)$. Finally, when $\ell \geq \mathrm{max}(m,n)$, by Corollary \ref{cor:mindegree}, we must have $\psdl{\ell}{K_{m,n}} = m+n$.
\end{proof}


\subsection{Trees}

For trees, it is well known that $Z_+(T) = 1$ which is in contrast to standard forcing which gives $Z(T) = P(T)$, the path cover number of a graph \cite{IEPGZF}. With the addition of leaks, we have the following theorem, which, for all $\ell\ge 1$, is surprisingly the same as for the $\ell$-leaky standard zero forcing \cite{DillmanKenter,LeakyReslience}. 
\begin{thm}\label{thm:LeakyTrees}
    Let $T$ be a tree, and $B = \{v\in V(T): \text{deg}(v)\leq \ell\}$. Then,

\begin{align*}
    \psdl{\ell}{T}&=\begin{cases}
    1 & \ell = 0, \\
    |B| & \  \ell \geq 1.
    \end{cases}
\end{align*}
\end{thm}

\begin{proof}

    Let $T$ be a tree, $\ell \ge 1$, and $B\subseteq V(T)$ be such that $B = \{v\in V(T): \text{deg}(v)\leq \ell\}$. By Theorem \ref{thm:degree}, we must have $|B| \le \psdl{\ell}{T}$. From \cite{LeakyReslience} and Proposition \ref{prop:psdlleqleaky}, we have $\psdl{\ell}{T} \le |B|$ and the result follows.


\end{proof}

\section{Cartesian products}\label{sec:CartesisnProducts}
For positive semidefinite zero forcing, we know $Z_+(G\Box G') \le \min\{n' Z_+(G), nZ_+(G')\} $ \cite{IEPGZF}. A similar result follows with the addition of $\ell$-leaks, since still taking the same vertices associated with an $\ell$-leaky forcing positive semidefinite set for $G$ will force $G\Box G'$ with $\ell$ leaks (similarly, for $G'$). 

\begin{thm}\label{thm:LeakyInequalityProduct}
    Let $G$ be a graph. Then, 
    $$Z_{(\ell)}^+(G\Box G') \le \min\{n' Z_{(\ell)}^+(G), nZ_{(\ell)}^+(G')\}. $$
\end{thm}




In the subsections that follow, we give the $\ell$-leaky positive semidefinite forcing number of families of graphs that can be realized as the Cartesian product of two graphs. 

\subsection{Grid graphs}
The grid graph $P_n\Box P_m$ can be viewed as points in an integer lattice in the first quadrant of the coordinate plane: $(i,j)$ such that $0\le i \le n-1$ and $0\le j \le m-1$. 

\begin{obs}
    Since all vertices in the grid graph have degree three or less, by by Corollary \ref{cor:mindegree}, we have $\psdl{\ell}{P_n\Box P_m} = n\cdot m$ for all $\ell\ge 3$. 
\end{obs}

Thus, the interesting cases are when $\ell = 1,2$. In \cite{DillmanKenter}, an upper bound for various $n,m$ are given on the $\ell$-leaky standard forcing number for $P_n\Box P_m$. For $\ell$-leaky positive semidefinite forcing, we are able to provide the following equality for particular $n$ and $m$ values. 
\begin{thm}
    For $n,m\in \mathbb{Z}_+$ such that $n\ge 4$ and $n=m$ or $n= m+1$, $\psdl{1}{P_n \Box P_m}= \min\{n,m\} = n.$
\end{thm}
\begin{proof}
    Consider $n = m, m+1$ and the subset of vertices $B\subset V(P_n\Box P_m)$, $B =  \{(0,0), (1,1), ... , (n-1,n-1)\}$. This set can be seen to force the entire graph, despite one leak. Therefore, we have $\psdl{1}{P_n \Box P_m} \le n$, and since $\psdl{0}{P_n \Box P_m} = n$, using Theorem \ref{thm:Inequalities}, the result follows. 


\end{proof}


\subsection{Generalized Peterson Graph $G(n,1)$}

Continuing with the Cartesian product of graphs, the generalized Peterson graph, $G(n,1)$, can be realized as $C_n\square P_2$, and visualized as a three-dimensional prism. 

\begin{thm}
    For the generalized Peterson graph, $GP(n,1)$, we have 

\begin{align*}
    \psdl{\ell}{GP(n,1)}&=\begin{cases}
        3 & \ell = 0,1, \ n=3, \\
        4 & \ell =0,1 \ n\geq 4,\\
        4 & \ell = 2 \ n= 4,5,6, \\
        2n & \ell\geq 3. \\
    \end{cases}
\end{align*}
\end{thm}

\begin{proof}
    Notice, when $\ell = 0$, $\psdl{0}{GP(n,1)} = 3 , 4$ when $n = 3$, $n\ge 4$ respectively \cite{PosSemiGeneralizedPeterson}. When $\ell = 1$ and $n=3$, it can be seen that $\psdl{1}{GP(n,1)} = 3$ using vertices of the 3-cycle. Using Theorem \ref{thm:Inequalities} and Theorem \ref{thm:LeakyInequalityProduct}, we have $\psdl{1}{GP(n,1)} =  4$ for all $n\ge 4$. Since $\delta(GP(n,1)) = 3$, we must have $\psdl{\ell}{GP(n,1)} = 2n$ when $\ell\ge 3$ by Corollary \ref{cor:mindegree}.




   Therefore, the interesting and final case is when $\ell = 2$ and $n \geq 3$. When $n=4,5,6$, coloring two vertices on the outside cycle a distance of at most two apart and their neighbor on the inside cycle blue will be a forcing set. Thus, using Theorem \ref{thm:Inequalities}, we have $\psdl{2}{GP(n,1)} = 4$.

\end{proof}

\begin{figure}[h!]
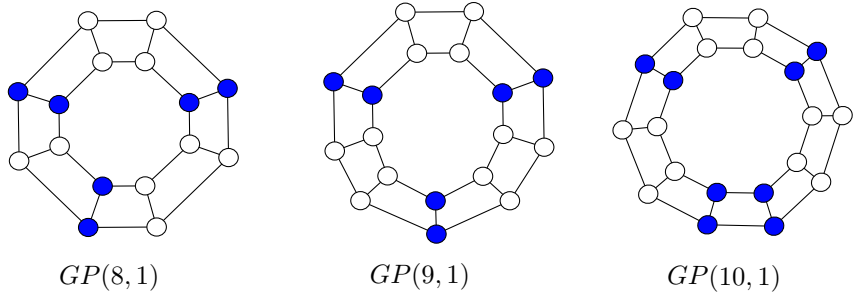

    \centering


    \caption{Coloring pattern for $GP(n,1)$ for two leaks}
    \label{fig:GPColoringPattern}
\end{figure}

When $\ell = 2$ and $n\ge 7$, we believe $\psdl{2}{GP(n,1)} = 2\displaystyle\left\lceil\frac{n}{3}\right\rceil$, although we were only able to provide the upper bound as in the following proposition. 

\begin{prop}
For $n\ge 7$, $\psdl{2}{GP(n,1)} \le 2\displaystyle\left\lceil\frac{n}{3}\right\rceil $.
\end{prop}
\begin{proof}
Let $n\ge 7$ and consider $\ell=2$ leaks on the graph $GP(n,1)$. We will color a subset of vertices as follows. On the outside cycle, color vertices blue ensuring as much as possible the maximum distance between any two blue vertices on the outside cycle is two. For each blue vertex now colored on the outside cycle, color the adjacent vertex on the inner cycle blue as well. See Figure \ref{fig:GPColoringPattern} for examples of this coloring pattern. Coloring in this way, we will color $\left\lceil\frac{n}{3}\right\rceil$ vertices on each cycle blue, and thus, $\psdl{2}{GP(n,1)} \le 2\displaystyle\left\lceil\frac{n}{3}\right\rceil$ as desired.  
\end{proof}

\subsection{$\mathbf{d}$-dimensional Hypercubes, $\mathbf{Q_d}$}

The $d$-dimensional hypercube $Q_d$ can be viewed as the Cartesian product $\underbrace{K_2\Box K_2 \Box \cdots \Box K_2}_{d}$. It can also be viewed as the graph with the set of vertices $V(Q_d)=\{0,1\}^d$ and edges $$E(Q_d)=\{(i,j):i,j\in V(Q_d) \text{ and } i, j \text{ differ by exactly one coordinate}\}.$$ See Figure \ref{fig:Hypercube} as an example of a $4$-dimensional hypercube. We note that for any $d$-dimensional hypercube $Q_d$, $|V(Q_d)|=2^d$ and the degree of all vertices in $Q_d$ is $d$.

\begin{figure}[h!]
    \centering
\begin{tikzpicture}[
    line width=0.6pt,
    every node/.style={circle, draw, thick, minimum size=6pt, inner sep=0pt, font=\scriptsize\bfseries}]

    \pgfsetxvec{\pgfpoint{0.9cm}{0.0cm}}
    \pgfsetyvec{\pgfpoint{0.0cm}{0.9cm}}

    \foreach \point / \id / \angle in {
        (0,0)/0001/270,
        (0,5)/0011/90,
        (5,0)/1001/270,
        (5,5)/1011/90,
        (2,2)/0101/180,
        (2,7)/0111/90,
        (7,2)/1101/270,
        (7,7)/1111/90,
        (2.5,1.5)/0000/270,
        (2.5,3.5)/0010/90,
        (4.5,1.5)/1000/250,
        (4.5,3.5)/1010/120,
        (3.5,2.5)/0100/170,
        (3.5,4.5)/0110/180,
        (5.5,2.5)/1100/10,
        (5.5,4.5)/1110/0}
    {
        \node (\id) at \point [label=\angle:\id] {};
    }


     \foreach \id in {0011, 1111, 0101, 1001, 1010, 0110, 0000, 1100}
     {
         \node[fill, blue] at (\id){};
     }


     \foreach \id in {0111, 1011, 0001, 1101, 1000, 1110, 0010, 0100}
     {
        \node[thick] at (\id){};
     }

    \path 
    (0011) edge (1011) edge (0111) edge (0001)
    (1001) edge (0001) edge (1101) edge (1011)
    (1111) edge (1101) edge (1011) edge (0111)
    (0010) edge (1010) edge (0110) edge (0000)
    (1000) edge (0000) edge (1100) edge (1010)
    (1110) edge (1100) edge (1010) edge (0110);

    \path[dashed]
    (0101) edge (1101) edge (0001) edge (0111)
    (0100) edge (1100) edge (0000) edge (0110);
    
    \path[dotted]
    (0000) edge (0001)
    (0010) edge (0011)
    (0100) edge (0101)
    (0110) edge (0111)
    (1000) edge (1001)
    (1010) edge (1011)
    (1100) edge (1101)
    (1110) edge (1111);
\end{tikzpicture}
    \caption{$4$-dimensional hypercube $Q_4$, with $\mathcal{E}$ colored blue.}
    \label{fig:Hypercube}
\end{figure}
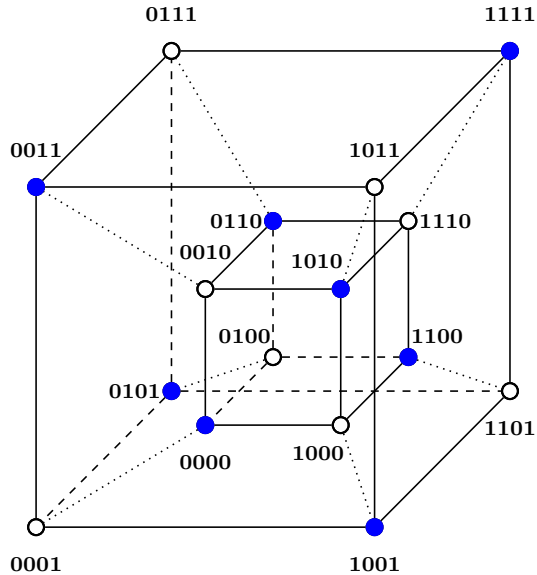



\begin{thm}
    For the $d$-dimensional hypercube $Q_d$,
\begin{align*}
    \psdl{\ell}{Q_d}&=\begin{cases}
        2^{d-1} & \ell\leq d-1, \\
        2^d & d\leq\ell\leq 2^d.
    \end{cases}
\end{align*}
\end{thm}
\begin{proof}
     From \cite{Hypercubes}, it is known that $d$-dimensional hypercubes $Q_d$ are bipartite graphs, with two parts $\mathcal{E}$ and $\mathcal{O}$ corresponding to the vertices with an even number of 1's and odd number of 1's, respectively. Additionally, note that $|\mathcal{E}|=|\mathcal{O}|=2^{d-1}=\psdl{0}{Q_d}$. 

     Considering $\ell \le d-2$, from\cite{D-2LeakyCube}, Proposition \ref{prop:psdlleqleaky}, and Theorem \ref{thm:Inequalities}, we have that $\psdl{\ell}{Q_d}=2^{d-1}$ for all $\ell\leq d-2$. In the case of $\ell=d-1$, let the initial set of blue vertices $B$ be equal to the set $\mathcal{E}$ (see Figure \ref{fig:Hypercube} colored in this manner). Since $Q_d$ is bipartite, every non-blue vertex will have only blue neighbors, and every non-blue vertex would be isolated into its own component $W_i$ of $Q_d \setminus B$ for all $i\in[2^{d-1}]$. In each $W_i$, since the non-blue vertex has degree $d$, we are able to force. Thus, with this and Theorem \ref{thm:Inequalities}, we have $\psdl{\ell}{Q_d}=2^{d-1}$ for all $\ell\leq d-1$. 
    

    In the case of $d\leq \ell\leq 2^d$, we refer again to the fact that all vertices in $Q_d$ have degree $d$. By Lemma \ref{thm:degree}, all vertices in $Q_d$ must thus be blue in order to force against $\ell$ leaks.
\end{proof}

    

\section{Positive semidefinite leaky forts}\label{sec:Forts}

\quad \quad Noticing similarities of many results between $\ell$-leaky standard and $\ell$-leaky positive semidefinite forcing, most notably Theorem \ref{thm:LeakyTrees} giving the equality of leaky forcing numbers for trees, we investigated forts and their interaction with leaks. Forts for standard forcing and positive semidefinite forcing were studied in \cite{FortsOG,FortsPSD}, respectively, with $\ell$-leaky standard forts first introduced in \cite{DillmanKenter} and used as a computational tool to calculate the $\ell$-leaky forcing number of a graph. We extend the notion of a positive semidefinite fort to $\ell$-leaky positive semidefinite forts with the purpose of viewing leaky forcing through a different lens in hopes of unveiling differences between the two types of leaky forcing.

\begin{defn}
\label{psdlfort}
    An \emph{$\ell$-leaky positive semidefinite fort} in a graph $G$ is a subset of vertices, $\psdlfort\subset V(G)$, such that, for each connected component $F_i^+$ in $G[F_{(\ell)}^+]$, and for each $v\not\in \psdlfort$, either $\text{deg}_{F_i^+}(v)=0$, $\text{deg}_{F_{i}^+}(v)\geq 2$, or $|\{ v\in V(G)\setminus\psdlfort : |N_G(v)\cap F_i^+|=1 \}|\leq \ell$.
\end{defn}

\begin{rem}
    It can be see that our definition of $\ell$-leaky positive semidefinite forts captures the notion of a positive semidefinite fort taking $\ell = 0$. See Figure \ref{fig:LeakyPSDForts} for an image of a general $\ell$-leaky positive semidefinite (psd) fort and an example of a $\ell$-leaky psd fort for the graph $GP(8,1)$ and 2 leaks.  
\end{rem}

\begin{figure}[h!]
    \centering
\begin{tikzpicture}[x=0.75pt,y=0.75pt,yscale=-1,xscale=1]

\draw    (138,109.83) -- (182.58,94.08) ;
\draw    (139,98.83) -- (182.58,94.08) ;
\draw   (52,149) .. controls (52,108.68) and (101.02,76) .. (161.5,76) .. controls (221.98,76) and (271,108.68) .. (271,149) .. controls (271,189.32) and (221.98,222) .. (161.5,222) .. controls (101.02,222) and (52,189.32) .. (52,149) -- cycle ;
\draw    (160,171) .. controls (185,192) and (139,198) .. (164,222) ;
\draw    (157,123) .. controls (182,144) and (135,147) .. (160,171) ;
\draw    (157.67,76.33) .. controls (182.67,97.33) and (132,99) .. (157,123) ;
\draw   (81.81,145.44) .. controls (77.91,137.01) and (82.15,126.29) .. (91.28,121.51) .. controls (100.41,116.73) and (110.98,119.68) .. (114.88,128.11) .. controls (118.78,136.55) and (114.54,147.26) .. (105.41,152.04) .. controls (96.28,156.83) and (85.71,153.87) .. (81.81,145.44) -- cycle ;
\draw   (110.48,112.77) .. controls (106.58,104.34) and (110.82,93.63) .. (119.95,88.84) .. controls (129.08,84.06) and (139.65,87.02) .. (143.55,95.45) .. controls (147.45,103.88) and (143.21,114.59) .. (134.08,119.38) .. controls (124.95,124.16) and (114.38,121.21) .. (110.48,112.77) -- cycle ;
\draw   (115.98,206.61) .. controls (112.08,198.18) and (116.32,187.46) .. (125.45,182.68) .. controls (134.58,177.89) and (145.15,180.85) .. (149.05,189.28) .. controls (152.95,197.71) and (148.71,208.43) .. (139.58,213.21) .. controls (130.45,218) and (119.88,215.04) .. (115.98,206.61) -- cycle ;
\draw  [fill={rgb, 255:red, 255; green, 255; blue, 255 }  ,fill opacity=1 ] (179.33,94.08) .. controls (179.33,92.29) and (180.79,90.83) .. (182.58,90.83) .. controls (184.38,90.83) and (185.83,92.29) .. (185.83,94.08) .. controls (185.83,95.88) and (184.38,97.33) .. (182.58,97.33) .. controls (180.79,97.33) and (179.33,95.88) .. (179.33,94.08) -- cycle ;
\draw    (248.72,130.84) -- (232.02,115.36) ;
\draw    (228.35,137.46) -- (232.02,115.36) ;
\draw  [fill={rgb, 255:red, 255; green, 255; blue, 255 }  ,fill opacity=1 ] (231.68,118.59) .. controls (229.9,118.4) and (228.6,116.8) .. (228.79,115.02) .. controls (228.98,113.23) and (230.58,111.94) .. (232.36,112.12) .. controls (234.15,112.31) and (235.44,113.91) .. (235.25,115.69) .. controls (235.07,117.48) and (233.47,118.78) .. (231.68,118.59) -- cycle ;
\draw    (113.33,130) -- (192.92,129.58) ;
\draw  [fill={rgb, 255:red, 255; green, 255; blue, 255 }  ,fill opacity=1 ] (189.67,129.58) .. controls (189.67,127.79) and (191.12,126.33) .. (192.92,126.33) .. controls (194.71,126.33) and (196.17,127.79) .. (196.17,129.58) .. controls (196.17,131.38) and (194.71,132.83) .. (192.92,132.83) .. controls (191.12,132.83) and (189.67,131.38) .. (189.67,129.58) -- cycle ;
\draw    (107.5,145) -- (192.92,145.25) ;
\draw  [fill={rgb, 255:red, 255; green, 255; blue, 255 }  ,fill opacity=1 ] (189.67,145.25) .. controls (189.67,143.46) and (191.12,142) .. (192.92,142) .. controls (194.71,142) and (196.17,143.46) .. (196.17,145.25) .. controls (196.17,147.04) and (194.71,148.5) .. (192.92,148.5) .. controls (191.12,148.5) and (189.67,147.04) .. (189.67,145.25) -- cycle ;
\draw    (140.33,187) -- (193.25,186.58) ;
\draw  [fill={rgb, 255:red, 255; green, 255; blue, 255 }  ,fill opacity=1 ] (190,186.58) .. controls (190,184.79) and (191.46,183.33) .. (193.25,183.33) .. controls (195.04,183.33) and (196.5,184.79) .. (196.5,186.58) .. controls (196.5,188.38) and (195.04,189.83) .. (193.25,189.83) .. controls (191.46,189.83) and (190,188.38) .. (190,186.58) -- cycle ;
\draw    (376.17,159.24) -- (350.36,168.74) ;
\draw    (403.55,184.96) -- (394.4,211.52) ;
\draw    (430.66,184.78) -- (438,211.23) ;
\draw    (459.15,157.94) -- (483.8,166.57) ;
\draw    (458.73,131.18) -- (483.13,122.05) ;
\draw    (430.57,104.9) -- (437.85,78.34) ;
\draw    (403.46,105.07) -- (394.25,78.63) ;
\draw    (375.76,132.48) -- (349.7,124.23) ;
\draw    (350.36,168.74) -- (394.4,211.52) ;
\draw    (483.8,166.57) -- (438,211.23) ;
\draw    (437.85,78.34) -- (483.13,122.05) ;
\draw    (349.7,124.23) -- (394.25,78.63) ;
\draw    (394.25,78.63) -- (437.85,78.34) ;
\draw    (483.8,166.57) -- (483.13,122.05) ;
\draw    (350.36,168.74) -- (349.7,124.23) ;
\draw    (394.4,211.52) -- (438,211.23) ;
\draw    (376.17,159.24) -- (403.55,184.96) ;
\draw    (459.15,157.94) -- (430.66,184.78) ;
\draw    (430.57,104.9) -- (458.73,131.18) ;
\draw    (375.76,132.48) -- (403.46,105.07) ;
\draw    (403.46,105.07) -- (430.57,104.9) ;
\draw    (459.15,163.77) -- (458.73,131.18) ;
\draw    (376.17,165.08) -- (375.76,132.48) ;
\draw    (403.55,184.96) -- (430.66,184.78) ;
\draw  [color={rgb, 255:red, 0; green, 0; blue, 0 }  ,draw opacity=1 ][fill={rgb, 255:red, 208; green, 2; blue, 27 }  ,fill opacity=1 ] (369.55,132.48) .. controls (369.55,129.26) and (372.33,126.65) .. (375.76,126.65) .. controls (379.18,126.65) and (381.96,129.26) .. (381.96,132.48) .. controls (381.96,135.71) and (379.18,138.32) .. (375.76,138.32) .. controls (372.33,138.32) and (369.55,135.71) .. (369.55,132.48) -- cycle ;
\draw  [fill={rgb, 255:red, 255; green, 255; blue, 255 }  ,fill opacity=1 ] (424.37,104.9) .. controls (424.37,101.67) and (427.15,99.06) .. (430.57,99.06) .. controls (434,99.06) and (436.78,101.67) .. (436.78,104.9) .. controls (436.78,108.12) and (434,110.73) .. (430.57,110.73) .. controls (427.15,110.73) and (424.37,108.12) .. (424.37,104.9) -- cycle ;
\draw  [fill={rgb, 255:red, 255; green, 255; blue, 255 }  ,fill opacity=1 ] (397.26,105.07) .. controls (397.26,101.85) and (400.04,99.23) .. (403.46,99.23) .. controls (406.89,99.23) and (409.67,101.85) .. (409.67,105.07) .. controls (409.67,108.29) and (406.89,110.91) .. (403.46,110.91) .. controls (400.04,110.91) and (397.26,108.29) .. (397.26,105.07) -- cycle ;
\draw  [fill={rgb, 255:red, 208; green, 2; blue, 27 }  ,fill opacity=1 ] (370.47,159.24) .. controls (370.47,156.02) and (373.24,153.41) .. (376.67,153.41) .. controls (380.1,153.41) and (382.87,156.02) .. (382.87,159.24) .. controls (382.87,162.46) and (380.1,165.08) .. (376.67,165.08) .. controls (373.24,165.08) and (370.47,162.46) .. (370.47,159.24) -- cycle ;
\draw  [color={rgb, 255:red, 0; green, 0; blue, 0 }  ,draw opacity=1 ][fill={rgb, 255:red, 255; green, 255; blue, 255 }  ,fill opacity=1 ] (452.53,131.18) .. controls (452.53,127.95) and (455.31,125.34) .. (458.73,125.34) .. controls (462.16,125.34) and (464.94,127.95) .. (464.94,131.18) .. controls (464.94,134.4) and (462.16,137.01) .. (458.73,137.01) .. controls (455.31,137.01) and (452.53,134.4) .. (452.53,131.18) -- cycle ;
\draw  [fill={rgb, 255:red, 255; green, 255; blue, 255 }  ,fill opacity=1 ] (452.94,157.94) .. controls (452.94,154.71) and (455.72,152.1) .. (459.15,152.1) .. controls (462.57,152.1) and (465.35,154.71) .. (465.35,157.94) .. controls (465.35,161.16) and (462.57,163.77) .. (459.15,163.77) .. controls (455.72,163.77) and (452.94,161.16) .. (452.94,157.94) -- cycle ;
\draw  [fill={rgb, 255:red, 255; green, 255; blue, 255 }  ,fill opacity=1 ] (424.46,184.78) .. controls (424.46,181.56) and (427.24,178.95) .. (430.66,178.95) .. controls (434.09,178.95) and (436.87,181.56) .. (436.87,184.78) .. controls (436.87,188.01) and (434.09,190.62) .. (430.66,190.62) .. controls (427.24,190.62) and (424.46,188.01) .. (424.46,184.78) -- cycle ;
\draw  [fill={rgb, 255:red, 255; green, 255; blue, 255 }  ,fill opacity=1 ] (397.35,184.96) .. controls (397.35,181.73) and (400.13,179.12) .. (403.55,179.12) .. controls (406.98,179.12) and (409.76,181.73) .. (409.76,184.96) .. controls (409.76,188.18) and (406.98,190.79) .. (403.55,190.79) .. controls (400.13,190.79) and (397.35,188.18) .. (397.35,184.96) -- cycle ;
\draw  [fill={rgb, 255:red, 208; green, 2; blue, 27 }  ,fill opacity=1 ] (343.5,124.23) .. controls (343.5,121) and (346.28,118.39) .. (349.7,118.39) .. controls (353.13,118.39) and (355.91,121) .. (355.91,124.23) .. controls (355.91,127.45) and (353.13,130.06) .. (349.7,130.06) .. controls (346.28,130.06) and (343.5,127.45) .. (343.5,124.23) -- cycle ;
\draw  [fill={rgb, 255:red, 208; green, 2; blue, 27 }  ,fill opacity=1 ] (388.05,78.63) .. controls (388.05,75.4) and (390.83,72.79) .. (394.25,72.79) .. controls (397.68,72.79) and (400.46,75.4) .. (400.46,78.63) .. controls (400.46,81.85) and (397.68,84.46) .. (394.25,84.46) .. controls (390.83,84.46) and (388.05,81.85) .. (388.05,78.63) -- cycle ;
\draw  [fill={rgb, 255:red, 255; green, 255; blue, 255 }  ,fill opacity=1 ] (431.65,78.34) .. controls (431.65,75.11) and (434.42,72.5) .. (437.85,72.5) .. controls (441.27,72.5) and (444.05,75.11) .. (444.05,78.34) .. controls (444.05,81.56) and (441.27,84.17) .. (437.85,84.17) .. controls (434.42,84.17) and (431.65,81.56) .. (431.65,78.34) -- cycle ;
\draw  [fill={rgb, 255:red, 255; green, 255; blue, 255 }  ,fill opacity=1 ] (476.93,122.05) .. controls (476.93,118.83) and (479.71,116.22) .. (483.13,116.22) .. controls (486.56,116.22) and (489.34,118.83) .. (489.34,122.05) .. controls (489.34,125.28) and (486.56,127.89) .. (483.13,127.89) .. controls (479.71,127.89) and (476.93,125.28) .. (476.93,122.05) -- cycle ;
\draw  [fill={rgb, 255:red, 255; green, 255; blue, 255 }  ,fill opacity=1 ] (477.59,166.57) .. controls (477.59,163.35) and (480.37,160.73) .. (483.8,160.73) .. controls (487.22,160.73) and (490,163.35) .. (490,166.57) .. controls (490,169.79) and (487.22,172.4) .. (483.8,172.4) .. controls (480.37,172.4) and (477.59,169.79) .. (477.59,166.57) -- cycle ;
\draw  [fill={rgb, 255:red, 255; green, 255; blue, 255 }  ,fill opacity=1 ] (431.79,211.23) .. controls (431.79,208.01) and (434.57,205.39) .. (438,205.39) .. controls (441.42,205.39) and (444.2,208.01) .. (444.2,211.23) .. controls (444.2,214.45) and (441.42,217.06) .. (438,217.06) .. controls (434.57,217.06) and (431.79,214.45) .. (431.79,211.23) -- cycle ;
\draw  [fill={rgb, 255:red, 208; green, 2; blue, 27 }  ,fill opacity=1 ] (388.7,211.52) .. controls (388.7,208.29) and (391.47,205.68) .. (394.9,205.68) .. controls (398.33,205.68) and (401.1,208.29) .. (401.1,211.52) .. controls (401.1,214.74) and (398.33,217.35) .. (394.9,217.35) .. controls (391.47,217.35) and (388.7,214.74) .. (388.7,211.52) -- cycle ;
\draw  [fill={rgb, 255:red, 208; green, 2; blue, 27 }  ,fill opacity=1 ] (344.16,168.74) .. controls (344.16,165.52) and (346.94,162.9) .. (350.36,162.9) .. controls (353.79,162.9) and (356.57,165.52) .. (356.57,168.74) .. controls (356.57,171.96) and (353.79,174.58) .. (350.36,174.58) .. controls (346.94,174.58) and (344.16,171.96) .. (344.16,168.74) -- cycle ;
\draw   (203.5,129) .. controls (208,144.28) and (212.5,153.44) .. (217,156.5) .. controls (212.5,159.56) and (208,168.72) .. (203.5,184) ;

\draw (152.67,54.07) node [anchor=north west][inner sep=0.75pt]    {$G$};
\draw (66,72.07) node [anchor=north west][inner sep=0.75pt]    {$F_{( \ell )}^{+}$};
\draw (88.04,126.23) node [anchor=north west][inner sep=0.75pt]    {$F_{2}^{+}$};
\draw (101.74,163.44) node [anchor=north west][inner sep=0.75pt]    {$\vdots $};
\draw (116.71,93.56) node [anchor=north west][inner sep=0.75pt]    {$F_{1}^{+}$};
\draw (122.21,187.39) node [anchor=north west][inner sep=0.75pt]    {$F_{i}^{+}$};
\draw (186.24,155.28) node [anchor=north west][inner sep=0.75pt]    {$\vdots $};
\draw (219.67,148.23) node [anchor=north west][inner sep=0.75pt]    {$\ell $};
\draw (66.5,226.9) node [anchor=north west][inner sep=0.75pt]  [font=\small]  {$( a) \ General\ \ell -leaky\ psd\ fort$};
\draw (323,225.9) node [anchor=north west][inner sep=0.75pt]  [font=\small]  {$( b) \ Example\ 2-leaky\ psd\ fort$};

\end{tikzpicture}

    \caption{}
    \label{fig:LeakyPSDForts}
\end{figure}

Theorem \ref{thm:iffpsdlforts} connects $\ell$-leaky positive semidefinite  forts with $\ell$-leaky positive semidefinite forcing sets. First, we need the definition of the closure of an initial set of colored vertices when there are $\ell$ leaks present in the graph \cite{Closure}. 
\begin{defn}
    The \emph{closure} of a set $B$ (denoted by \textit{cl(B)}) when $\ell$ leaks appear in the graph after the set $B$ has been determined is defined as the set of colored vertices after exhaustively applying the color change rule. The closure of a set of vertices will always be the same vertices, regardless of the process in which the graph is colored as proven in \cite{Closure}.
\end{defn}

\begin{thm}
\label{thm:iffpsdlforts}
    Let $G$ be a graph and $\ell\ge 0$. A set $B_{(\ell)}^+\subseteq V(G)$ is a $\ell$-leaky positive semidefinite forcing set if and only if $B_{(\ell)}^+$ intersects all $\ell$-leaky positive semidefinite forts.
\end{thm}

\begin{proof}
    
    First, towards a contradiction, assume that there exists a set $B_{(\ell)}^+\subset V(G)$ that intersects all $\ell$-leaky positive semidefinite forts in $G$, but it is not a $\ell$-leaky positive semidefinite forcing set. Since $B^+_{(\ell)}$ is not a forcing set, once all possible forces are carried out, it follows that $cl(B_{(\ell)}^+)\not=V(G)$, which means that $V(G)\backslash cl(B_{(\ell)}^+) = F$ is nonempty and does not contain any blue vertices. Denote $F_i$ as the connected components of $F$. For $v\in cl(B_{(\ell)}^+)$, since no other forces can occur in $G$, we must have $\text{deg}_{F_i}(v) = 0$, $\text{deg}_{F_i}(v) = 2$, or $ |\{ u\in N_G(v) \cap F_i| = 1  \} | \le \ell$ for all components $F_i$. Therefore, by definition, $F$ would be a $\ell$-leaky  positive semidefinite fort that does not intersect $B_{(\ell)}^+$, a contradiction, and thus $B_{(\ell)}^+$ must be an $\ell$-leaky positive semidefinite forcing set.  
    
    Next, again towards a contradiction, let us assume that there exists a set $B_{(\ell)}^+\subset V(G)$ that is a $\ell$-leaky positive semidefinite forcing set, but does not intersect all $\ell$-leaky positive semidefinite forts. In other words, assume that there exists a $\ell$-leaky positive semidefinite fort $F_{(\ell)}^+\subset V(G)$ such that $B_{(\ell)}^+\cap F_{(\ell)}^+=\varnothing$. It follows that $B_{(\ell)}^+\subseteq V(G)\setminus F_{(\ell)}^+$, meaning $cl(B_{(\ell)}^+)\subseteq cl(V(G)\setminus F_{(\ell)}^+)$. Since $F_{(\ell)}^+$ is a fort, $cl(V(G)\backslash F_{(\ell)}^+)\neq V(G)$. Thus, $F_{(\ell)}^+$ would be excluded from the closure of the initial forcing set, which means that $B_{(\ell)}^+$ could not actually have been a $\ell$-leaky positive semidefinite forcing set. A contradiction and the result follows.
\end{proof}

With Theorem \ref{thm:iffpsdlforts}, we can now use $\ell$-leaky  positive semidefinite forts to prove results about $\ell$-leaky positive semidefinite sets. The results given in Section \ref{sec:PropertiesLeakyForcing} can be reframed and proven in terms of $\ell$-leaky positive semidefinite forts. For example, Theorem \ref{thm:degree} can be restated as a corollary of Theorem \ref{thm:iffpsdlforts}.

\begin{cor}
    For any graph $G$ and $\ell\ge 0$, any positive semidefinite $\ell$-leaky forcing set has at least the vertices of degree $\ell$ or less.
\end{cor}


Our investigation of $\ell$-leaky positive semidefinite forts was motivated by the connections seen between $\ell$-leaky standard forcing and $\ell$-leaky positive semidefinite forcing. For many of the graph families we considered, we had $\leaky{\ell}{G}=\psdl{\ell}{G}$ for various $\ell$, with in some cases equality held for all $\ell$. Focusing on leaky forts, we investigated why this equality of leaky forcing numbers was the same for those graphs. We first made the following observation.

\begin{obs}
     If a positive semidefinite $\ell$-leaky fort $F_{(\ell)}^+$ of $G$, denoted $G[F_{(\ell)}^+]$, is connected, then that positive semidefinite $\ell$-leaky fort $F_{(\ell)}^+$ of $G$ is also a $\ell$-leaky standard fort of $G$. 
\end{obs}
With the previous observation in mind, we pose the following question. 
 \begin{question}
     Connected positive semidefinite $\ell$-leaky forts correspond with $\ell$-leaky standard forts. This is not enough to give $\leaky{\ell}{G}=\psdl{\ell}{G}$. What properties of the graph $G$ are sufficient or necessary for the equality of $\ell$-leaky forcing numbers? 
 \end{question}

\section*{Acknowledgements}

This work started at the Carnegie Mellon Summer Undergraduate Applied Mathematics Institute which took place in Summer 2023 with support from the
National Science Foundation under Grant No. 2244348. We also thank Dr. Michael Young, Dr. David Offner, and Dr. Juergen Kritschgau for their advice and support, along with the anonymous reviewers for their comments and suggestions.

\bibliography{references}

\begin{thebibliography}{10}

\bibitem{LeakyReslience}
J.~Alameda, J.~Kritschgau, N.~Warnberg, and M.~Young.
\newblock On leaky forcing and resilience.
\newblock {\em Discrete Applied Mathematics}, 306:32--45, 2022.

\bibitem{GeneralizationsLeaky}
J.~Alameda, J.~Kritschgau, and M.~Young.
\newblock Generalizations of leaky forcing.
\newblock 2020.

\bibitem{PSDOrigins}
F.~Barioli, W.~Barrett, S.~Fallat, {H. T.} Hall, L.~Hogben, B.~Shader, P.~{Van den Driessche}, and H.~{Van der Holst}.
\newblock Zero forcing parameters and minimum rank problems.
\newblock {\em Linear Algebra and its Applications}, 433:401--411, 2010.

\bibitem{Closure}
B.~Brimkov and J.~Carlson.
\newblock Minimal zero forcing sets.
\newblock 2022.

\bibitem{SearchTrees}
B.~Brimkov, C.~Fast, and I.~Hicks.
\newblock Computational approaches for zero forcing and related problems.
\newblock {\em European Journal of Operational Research}, 273:889--903, 2019.

\bibitem{Forts-Brimkov2019}
B.~Brimkov, C.~Fast, and I.~Hicks.
\newblock Computational approaches for zero forcing and related problems.
\newblock {\em European Journal of Operational Research}, 273(3):889--903, 2019.

\bibitem{QuantumControl}
D.~Burgarth, D.~{D'Alessandro}, L.~Hogben, S.~Severini, and M.~Young.
\newblock Zero forcing, linear and quantum controllability for systems evolving on networks.
\newblock {\em IEEE Transactions on Automatic Control}, 58:2349--2354, 2013.

\bibitem{DillmanKenter}
S.~Dillman and F.~Kenter.
\newblock Leaky forcing: A new variation of zero forcing.
\newblock {\em arXiv: 1910.00168}, 2019.

\bibitem{PSDforcingEdgeDelete}
Jason Ekstrand, Craig Erickson, H.~Tracy Hall, Diana Hay, Leslie Hogben, Ryan Johnson, Nicole Kingsley, Steven Osborne, Travis Peters, Jolie Roat, Arianne Ross, Darren~D. Row, Nathan Warnberg, and Michael Young.
\newblock Positive semidefinite zero forcing.
\newblock {\em Linear Algebra and its Applications}, 439(7):1862--1874, 2013.

\bibitem{FortsOG}
C.~Fast and I.~Hicks.
\newblock Effects of vertex degrees on the zero-forcing number and propagation time of a graph.
\newblock {\em Discrete Applied Mathematics}, 250:215--226, 2018.

\bibitem{Hypercubes}
D.~Galvin.
\newblock Independent sets in the discrete hypercube.
\newblock Expository note., January 2019.

\bibitem{AIMMINIMUMRANKWorkgroup}
AIM Minimum Rank – Special Graphs~Work Group.
\newblock Zero forcing sets and the minimum rank of graphs.
\newblock {\em Linear Algebra and its Applications}, 428(7):1628--1648, 2008.

\bibitem{D-2LeakyCube}
R.~Herrman.
\newblock The (d-2)-leaky forcing number of $q_d$ and $\ell$-leaky forcing number of $g(n, 1)$.
\newblock {\em Discrete Optimization}, 46, 2022.

\bibitem{IEPGZF}
L.~Hogben, J.~C.-H. Lin, and B.~Shader.
\newblock {\em Inverse Problems and Zero Forcing For Graphs}, volume 270.
\newblock Mathematical Surveys and Monographs, 2022.

\bibitem{PosSemiGeneralizedPeterson}
T.~Peters.
\newblock Positive semidefinite maximum nullity and zero forcing number.
\newblock {\em The Electronic Journal of Linear Algebra}, 23, 2012.

\bibitem{FortsPSD}
L.~Smith, D.~Mikeshell, and I.~Hicks.
\newblock An integer program for positive semidefinite zero forcing in graphs.
\newblock {\em Networks: An International Journal}, 76:366--380, 2020.

\bibitem{Forts-Optimal}
Logan~A. Smith and Illya~V. Hicks.
\newblock New computational approaches for the power dominating set problem: Set covering and the neighborhoods of zero forcing forts.
\newblock {\em Networks}, 79(2):202--219, 2022.

\end{thebibliography}

\end{document}